\documentclass[arxiv,onefignum,onetabnum]{siamart171218}

\usepackage{lipsum}
\usepackage{amsfonts}
\usepackage{graphicx}
\usepackage{epstopdf}
\usepackage{algorithmic}
\usepackage{amsmath,amssymb}
\usepackage{cite}
\usepackage{adjustbox}
\usepackage{blindtext}
\usepackage{CJK}
\usepackage{bm}
\usepackage[numbers,sort&compress]{natbib}

\usepackage{subfigure,booktabs,multirow}
\usepackage{mathrsfs,enumerate,color}
\usepackage{appendix}
\usepackage{float}
\usepackage{cite}

\usepackage{amsopn}

\newsiamremark{remark}{Remark}
\newsiamremark{hypothesis}{Hypothesis}
\crefname{hypothesis}{Hypothesis}{Hypotheses}
\newsiamthm{claim}{Claim}

\numberwithin{figure}{section}
\numberwithin{table}{section}
\headers{Adaptive randomized   algorithms for fixed-threshold low-rank matrix approximation }{Q. Liu and Y. Yu}

\title{Efficient adaptive randomized   algorithms   for fixed-threshold low-rank matrix approximation\thanks{
Submitted to the editors DATE.
\funding{The research is supported  by     Natural Science Foundation of Shanghai under grant 23ZR1422400 and  Large-scale Numerical Simulation Computing Sharing Platform of Shanghai University}}
}

\author{ Qiaohua Liu\thanks{Corresponding author. Department of Mathematics, Shanghai University and Newtouch Center for Mathematics of Shanghai University, Shanghai 200444, People's Republic of China(\email{qhliu@shu.edu.cn}).}
\and Yuejuan Yu\thanks{Department of Mathematics, Shanghai University, Shanghai 200444, People's Republic of China.}
}

\begin{document}
\maketitle

\begin{abstract} The low-rank matrix approximation problems  within a threshold    are widely applied in
information retrieval, image processing, background estimation of the video sequence problems and so on.
This paper presents an adaptive randomized rank-revealing algorithm of the data matrix $A$, in which
the basis matrix $Q$ of the approximate range space is adaptively  built block by block, through a recursive deflation procedure
on $A$.   Detailed analysis of randomized projection schemes are provided to analyze the numerical rank reduce during the deflation. The provable spectral and Frobenius error
$(I-QQ^T)A$ of the approximate
 low-rank matrix  $\tilde A=QQ^TA$  are presented, as well as  the approximate singular values.
This blocked deflation technique is  pass-efficient and can accelerate  practical  computations of large matrices.
Applied to image processing and background estimation problems, the blocked randomized algorithm   behaves more reliable and more
efficient than the known Lanczos-based method and a rank-revealing algorithm proposed by  Lee, Li and Zeng (in
SIAM J. Matrix Anal. Appl. 31 (2009), pp. 503-525).

\end{abstract}



\section{Introduction} \label{intro-sec}

 In many   image processing, information retrieval  and scientific computing applications,  it is of great importance to compute a low-rank approximation
 to a large data matrix.
 Usually, the problem of low-rank matrix approximation falls into two categories:

 $\bullet$ The fixed-rank problem, where the rank parameter $k$ is given;

 $\bullet$ The fixed-threshold, where the approximation error measured in Frobenius norm or spectral norm    might be as small as possible   within a given accuracy tolerance.

The singular value decomposition (SVD,\cite{gv2}) is undoubtedly the most reliable method for determining the optimal  approximation.
 To be precise, given
 an $m\times n$ ($m\ge n$) real  matrix $A$ with the    SVD:
\begin{equation}
 A=U\Sigma V^T=[U_{(k)}\quad U_{(k)\bot}]\left[\begin{array}{cc} \Sigma_{(k)}&\\ & \Sigma_{(k)\bot}\end{array}\right]\left[\begin{array}c V_{(k)}^T\\ V_{(k)\bot}^T\end{array}\right],\label{1.0}
 \end{equation}
where $\Sigma_{(k)}\in {\mathbb R}^{k\times k}, \Sigma_{(k)\bot}\in {\mathbb R}^{(m-k)\times (n-k)}$, the columns of $U_{(k)}$ and $U_{(k)\bot}$ are the corresponding left singular vectors, and the columns of $V_{(k)}$ and $V_{(k)\bot}$  are the corresponding right singular vectors. If we arrange   the singular values  of $A$  in decreasing order, then the optimal rank-$k$ approximation $A_k$ is given by
\begin{equation}
 A_k=U_{(k)}\Sigma_{(k)}V_{(k)}^T,\label{1.1}
 \end{equation}
 If further there is a significant gap between $\sigma_k$ and $\sigma_{k+1}$ such that
\begin{equation}
 \sigma_1\ge \sigma_2\ge \cdots\ge \sigma_k>\theta>\sigma_{k+1}\ge \cdots\ge \sigma_n,\label{1.2}
 \end{equation}
then, with the given threshold  $\theta$,   $A_k$ is also the optimal approximation matrix  such that $\|A-A_k\|_2=\|(I-U_{(k)}U_{(k)}^T)A\|_2<\theta$.
Within the threshold $\theta$,  the span of the columns of $U_{(k)}$ is denoted as the numerical range  ${\cal R}_\theta (A)$ of $A$, and $k$ is called the numerical rank of $A$.
We refer to  it as  ${\rm rank}_\theta(A)=k$.

Designing good matrix approximations with a given threshold $\theta$ arises in applications such as  recommendation systems, face recognitions, for which a convex optimization problem \cite{ccs}
\begin{equation}
\min_X {1\over 2}\|X-A\|_F^2+\tau \|X\|_*\label{1.3}
\end{equation}
needs to be solved,
where the nuclear norm $\|X\|_*$  is defined by the sum of singular values of $X$.  The optimal
approximation
\begin{equation}
\hat X={\sf shrink}({A}, \tau):=U{\mathscr S}_\tau(\Sigma)V^T\label{shrk}
\end{equation}
 can be obtained through the   the singular value thresholding (SVT) method, where
${\mathscr S}_\tau(\Sigma)$ is the element-wise application of the soft-thresholding operator
defined as
\begin{equation}
{\mathscr S}_\tau(x)={\sf sgn}(x){\rm max}(|x|-\tau, 0).\label{0002}
\end{equation}

For  matrix approximation and SVT problems, the truncated SVD is optimal but very expensive when the data matrix is large.
One can apply  the iterative Lanczos algorithm \cite{la,ws} ({\sf lansvd}) to compute
only the $k$ dominant singular values and singular vectors of large matrices. The {\sf lansvd}  can be implemented via a well known high-quality package  PROPACK \cite{la}. However, PROPACK cannot automatically compute only those singular values exceeding
the threshold $\theta$. The goal of this paper is to develop
adaptive matrix approximation algorithms with a given threshold. The approximate numerical range and numerical rank can also be
derived from the proposed algorithm.

In the literature, the latest study on adaptive matrix approximation problems with a  fixed-threshold are randomized
algorithms given in \cite{hmt,mv,ygl}.  Randomized methods  \cite{gu,hmt,llj,lwm,ma, mart,mv,ygl,rlb,rxb,wlr} for low-rank matrix approximations have attracted significant attention in the past decade.  Compared to traditional deterministic methods,  randomized methods are more efficient in that  the algorithms are based on the matrix-matrix products and QR factorizations that have been highly optimized for maximum efficiency on modern serial and parallel architectures.
 The  prototype of randomized algorithms in \cite{hmt} uses random projection to produce a kind of
 $QB$ {\it approximation}:
  $A\approx QB,$    where
$Q$ is  an $m\times (k+p)$ orthonormal basis of the range of $Y=A\Omega$. Here $\Omega$ is an $n\times (k+p)$  random Gaussian matrix, $p$ is an oversampling parameter to enhance the approximation quality of ${\cal R}(Q)$ to the dominant subspace of  $A$,  and  $B=Q^TA$ is a $(k+p)\times n$ matrix. Then, standard factorization, e.g., SVD, can be further performed on the smaller matrix $B$, to obtain a low-rank approximation of $A$. Usually, the numerical $k$ is assumed to be known in advance.

 An adaptive randomized range finder (ARRF)  \cite{hmt} uses a high-probability residual spectral norm estimate of  $(I-QQ^T)A$:
\[
\|(I-QQ^T)A\|_2\le 10 \sqrt{2\over \pi}\max_{i=1,\ldots,r} \|(I-QQ^T)A\omega^{(i)}\|_2
\]
such that  $\|(I-QQ^T)A\|_2<\theta$,  where  $\{\omega^{(i)}, i=1, 2,\ldots, {r}\}$
is  an independent family  of standard Gaussian vectors.
 Two blocked variants of  the randomized QB ({\sf randQB}) by \cite{mv,ygl}
are essentially the blocked Gram-Schmidt scheme of $A\Omega$, by building $Q:=[Q~~Q_i]$ incrementally and  measuring  the approximation error via explicitly/implicitly  maintaining the residual matrix $A:=A-Q_iB_i$ with $B_i:=Q_i^TA$ until $\|A\|_F$ is smaller than the given precision. Through various experimental results, this scheme improves the non-ideal situation of ARRF \cite{hmt}  for estimating the approximation error and the computed rank.
Nevertheless, there are some applications that the data matrix is contaminated by the noise  and  $A$ might have a ``heavy-tail'' of
singular values.  For example, a matrix with its top $k$ squared singular values all equal to 10 followed by a tail of smaller
singular values (e.g. 1000$k$ at 1)  makes the optimal spectral  approximation error $\|A-A_k\|_2^2=1$, while the Frobenius error $\|A-A_k\|_F^2=1000k$ might be very large.
  This renders meaningless of the stopping criterion in terms of the Frobenius norm.

In addition to the mentioned algorithms in \cite{hmt,mv,ygl}, there are two existing codes for rank-revealing algorithms for some particular singular values approximations:
the UTV tools implemented by Fierro, Hansen and Hansen \cite{hh}, and  a rank-revealing (RankRev) package proposed by Lee, Li and Zeng \cite{llz,llz2}.
The {\sf lulv} in UTV and {\sf larank} in RankRev are applicable for  low-rank approximation problems.
 Both {\sf lulv} and {\sf larank}
calculate some particular singular vectors  approximation by performing  the power iteration on $LL^T$ (or $AA^T$) with a random initial vector, while the former
computes a thin factorization $A=ULV^T$ and  maintains the unitary or triangular  structures of   $U, L, V$ with dimensions not less than $n\times n$ from a to z,  and
the later {\sf larank}  only computes the basis of the numerical range space, and can be capable
of handling larger matrices due to much less memory requirement and computational cost.
The idea of {\sf larank}  first finds   a coarse approximation  $\tilde u_1$  of the left singular vector $u_1$, followed by the  computation of a deflated matrix $\tilde A=A-\tilde u_1\tilde u_1^TA$.  The rank of the deflated matrix will be reduced by 1 in each deflation, and the application of the power scheme on $\tilde A\tilde A^T$
 yields a unit vector $\tilde u_2$ in ${\cal R}_\theta(A)$ with $\tilde u_2\bot \tilde u_1$.
Applying the deflation process recursively,
  approximate singular values above the threshold can be obtained one by one,   along with
the approximate rank of the matrix and an orthonormal basis for the approximate
range space. However, no provable results of approximation error and  estimated singular values are provided.

{\it A. Our contributions}

Motivated by the work in \cite{llz} and the frame work of the randomized QB algorithm, we aim to develop a blocked randomized rank-revealing algorithm within  a given threshold $\theta$. In the algorithm, the  residual matrix $A:=(I-Q_iQ_i^T)A$ is implicitly formed and the approximate numerical
range space is characterized by
\[
{\cal R}_\theta(A)=\mbox{span}\{Q_1, Q_2,\ldots, Q_s\},
\]
 where each $Q_j$ for $j<s$ has $b$ columns, while the columns of $Q_s$ depend on the properties of $A$, the  threshold $\theta$ and the stopping criterion.

A bottleneck in designing   rank-revealing algorithms is to track  singular values above the threshold $\theta$ with acceptable accuracy.
A key problem is that when the deflation process is utilized, does the deflated process $A^{(i)}=(I-Q_iQ_i^T)A^{(i-1)}$  with $A^{(0)}=A$ maintain  $A$'s tail singular values?
If not, applying the fixed threshold $\theta$ to the deflated matrix $A^{(i)}$ will be invalid to determine the right numerical rank of the original matrix $A$.
If yes, how to approximate those singular values above $\theta$, and how to estimate the   rank decrease during the deflation step?
Moreover, in the randomized algorithm (see \cite{gu,hmt} or Theorem \ref{thm:2.1}), the start Gaussian matrix with oversampling columns can generally
guarantee a high-quality basis matrix $Q$.  If  our algorithm is free of oversampling, what factors will affect the singular value approximation? These   core problems will be studied in this paper.

 We combine the randomized QB algorithm with the deflation process to design effective rank-revealing algorithms.
 The statistical properties of  the initial  Gaussian matrix  and singular value gap ratios of $A$ are proven to be crucial to the stability of the algorithm. In addition to this, we prove that
with a fixed block size $b$ and some mild assumptions,  the deflated matrix $(I-Q_1Q_1^T)A$ has rank  $b$ less than that of $A$ and  can  still  maintain $n-b$ singular values of $A$ with acceptable absolute error. The later deflation process also preserves the property and the singular values below $\theta$ have been maintained with acceptable accuracy.
This is the guarantee for the effectiveness and robustness of our proposed algorithm.
 We also provide synthetic data and real applications in information retrieval,  background estimation of  video  to demonstrate the efficiency and effectiveness of our proposed algorithm.

{\it B. Notation}

Let ${\mathbb R}^{m\times n}$   be the set of all $m\times n$ real matrices, and denote by ${\cal R}(A), {\cal K}(A)$ the column range and kernel space of $A$, respectively.
${\cal R}_\theta(A)$ and ${\cal K}_\theta(A)$ denote  the range and the kernel spaces are determined with respect to a threshold $\theta$. For example, with the notation of
\eqref{1.0} and \eqref{1.2}, ${\cal K}_\theta (A)={\rm span}\{v_{k+1},\ldots, v_n\}$.
Let $\|\cdot\|_{2,F}$ denote the spectral and Frobenius norm, respectively, and let $\|\cdot\|_*$ and {\sf shrink}($\cdot$) be the nuclear norm and  shrinkage operation defined in \eqref{1.3}-\eqref{shrk}.

Given an $m\times n$  matrix $A$,  $A^\dag$ means the Moore-Penrose inverse of $A$, and $\sigma_{\rm max}(A)$(or $\sigma_{\rm min}(A)$)  and $\sigma_j(A)$ mean the largest (or smallest) and the $j$th largest singular value of $A$, respectively. Let the function $Q={\sf orth}(A)$ denote the orthonormalization of the columns of $A$, which can be achieved by
 calling a packaged QR factorization, say in Matlab, we call [$Q$, $\sim$ ]={\sf qr}($A$,0). If $Q$ has orthonormal columns,
 $Z={\sf reorth2}(Y-Q(Q^TY))$ means
\[
Y:=Y-Q(Q^TY),\quad  Y:={\sf orth}(Y),\quad Z= Y-Q(Q^T Y).
\]

 For a symmetric matrix $M$, the function $[V,D]={\sf eig_\downarrow}(M)$ means we compute the eigenpair of $M$ with the eigenvalues in $D$  arranged in descending order, and the columns of $V$ being corresponding  eigenvectors.
 In addition, if $W_1\in {\mathbb R}^{n\times s}$, $Z_1\in {\mathbb R}^{n\times t}$ ($s\le t$) have orthonormal columns, then the projection onto the range of
 $W_1$ is defined as ${\cal P}_{W_1}=W_1W_1^T$, and  
  the deviation degree of subspace  ${\cal R}(W_1)$ with respect to   ${\cal R}(Z_1)$ is defined as
\begin{equation}
{\mathbf d}({\cal R}(W_1), {\cal R}(Z_1))=\|(I-{\cal P}_{Z_1}){\cal P}_{W_1}\|_2.\label{dist}
\end{equation}
If $s=t$, then ${\bf d}({\cal R}(W_1), {\cal R}(Z_1))=\|{\cal P}_{W_1}-{\cal P}_{Z_1}\|_2=\|(I-{\cal P}_{W_1}){\cal P}_{Z_1}\|_2.$ (see \cite[Theorem 2.5.1]{gv2}) 

\section{Blocked randomized algorithms}

In this section, we first introduce the basic randomized SVD algorithm and its blocked version.

\begin{algorithm}
\caption{(RSVD) Basic randomized SVD algorithm}  \label{rQB}
\begin{algorithmic}
\REQUIRE $m\times n$ matrix $A$ with $m\ge n$, integer $k>0$, oversampling parameter $p\ge 4$ and $\ell=k+p\ll n$.

\ENSURE  A rank-$k$ approximation $\widehat A_k$.

\noindent 1: Draw a random $n\times \ell$  Gaussian matrix $\Omega$ via $\Omega={\sf randn}(n,\ell)$.

2: Compute an orthogonal column basis  for $A\Omega$, we denote $Q={\sf orth}(A\Omega)$.

3: Compute $B=Q^TA$.

4: Compute the SVD: $B=\widehat U_B\widehat \Sigma\widehat V^T$.

5: Compute $\widehat U=Q\widehat U_{ B}$ and get $\widehat A_k=\widehat U_{(k)}\widehat\Sigma_{(k)}\widehat V_{(k)}^T$ from truncated SVD of $QB$.
\end{algorithmic}
\end{algorithm}

The algorithm with steps 1-3 is referred as the    {\sf randQB} algorithm in \cite{mv}. It produces an projector  $QQ^T$ such that $A\approx QQ^TA$.  The theoretical results in \cite[Corollaries 10.9-10.10]{hmt} demonstrate that
 the columns of $Q$ can capture very well the main features of the range space of $A$, and these approximations can be very  accurate
if oversampling is used (usually $p\ge 4$) and   $A$ has  very rapid decaying rate in its    singular values. When the singular values decay slowly, the power scheme
$Q={\sf orth}((AA^T)^qA\Omega)$ is recommended to amplify the gap and produces   the quality of column-orthonormal basis matrix $Q$.

%
%

Within the subspace iteration framework, Gu \cite{gu} presented new
low-rank approximation error as well as the accuracy for the singular values. The bounds are further refined by Saibaba \cite{sa}. In addition, canonical angles between the exact and approximate singular subspace are also analyzed to demonstrate the high quality of $\widehat U_{(k)}$ via the power scheme. For later analysis, we summarize these results below.

\begin{theorem}{\rm (\cite[Theorem 4, Theorem 8]{sa})}\label{thm:2.1}  Let the singular value ratio $\gamma_j=\sigma_{k+1}/\sigma_j$ for $j=1,2,\ldots,k$. Assume that $Q, \widehat U, B$ are generated via Algorithm 2.1  with power scheme $Q={\sf orth}((AA^T)^qA\Omega)$, and the matrix $V^T\Omega$ has the partition
\[
V^T\Omega=\left[\begin{array}c V_{(k)}^T\Omega\\ V_{(k)\bot}^T\Omega\end{array}\right]=\left[\begin{array}l
\Omega_1\\ \Omega_2\end{array}\right]\begin{array}l k\\ n-k\end{array}.
\]
Then
\[
\begin{array}l
{\bf d}({\cal R}(\widehat U_{(k)}),{\cal R}(U_{(k)}))\le {\displaystyle\gamma_k^{2q+1}\over\displaystyle 1-\gamma_k}\|\Omega_2\Omega_1^\dag\|_2,\\[1em]
\|(I-QQ^T)A\|_{2,F}^2\le \|\Sigma_{(k)\bot}\|_{2,F}^2+\gamma_k^{4q}\|\Sigma_{(k)\bot}\Omega_2\Omega_1^\dag\|_{2,F}^2,\\[1em]
\sigma_j(A)\ge \sigma_j(Q^TA)\ge \frac{\displaystyle\sigma_j(A)}{\displaystyle\sqrt{1+\gamma_j^{4q+2}\|\Omega_2\Omega_1^\dag\|_2^2}},
\end{array}
\]
where $\widehat U_{(k)}$ contains the first $k$ columns of $\widehat U$, and for $0<\delta\ll 1$, $\|\Omega_2\Omega_1^\dag\|_2\le {\cal C}_{\delta,k,\ell}$ with the probability at least $1-\delta$. Here ${\cal C}_{\delta,k,\ell}$ is defined in  Lemma \ref{lem:3.3}.
\end{theorem}

\begin{algorithm}  \caption{({\sf randQB\underline{~}b}) The blocked randomized QB algorithm \cite{mv}} \label{brQB}
\begin{algorithmic}[1]
\REQUIRE $m\times n$ matrix $A$ with $m\ge n$,  block size $b$ and fixed-precision $\theta$.

\ENSURE A blocked randomized QB approximation such that $\|A-QB\|_F<\theta$.

 \STATE Initialize $Q=[\quad]; B=[\quad]$; $E=\|A\|_F^2$.

\FOR{ $i=1,2,$ $\ldots$}         
     \STATE  Set $\Omega_i={\sf randn}(n,b)$.

     \STATE   Compute $Q_i={\sf orth}(A\Omega_i)$.

     \STATE Compute $Q_i={\sf orth}(Q_i-Q(Q^TQ_i))$.

     \STATE Set $B_i=Q_i^TA$.

     \STATE Update $A=A-Q_iB_i$.

   \STATE Set $Q=[Q\quad Q_i]$, $B=\Big[{B\atop B_i}\Big]$.

 \STATE If $\|A\|_F<\theta$ then stop.

\ENDFOR
 \end{algorithmic}
\end{algorithm}

 Algorithm \ref{brQB} is the blocked variant \cite{mv} of randomized QB process.
The blocked version is exactly the blocked  modified Gram-Schmidt scheme of $A[\Omega_1~ \ldots~ \Omega_i]$ with reorthognalization.
By   replacing step 4,  steps 7 and 9 with

4':  \quad $Q_i={\sf orth}(A\Omega_i-Q(B\Omega_i))$.

7': \quad $E=E-\|B_i\|_F^2$.

9': \quad   If $\|E\|_F<\theta$ then stop, \\
the storage and computational cost of the algorithm can be reduced, since  the updating of the  residual matrix in step 7 is avoided. The improved algorithm is called {\sf randQB\underline{~}EI} in \cite{ygl}.
Let $C_{\rm mm}$ and $C_{\rm qr}$ denote the scaling constants for the cost of executing a matrix-matrix multiplication and a full QR factorization, respectively.
Then
the two dense matrices of size $m\times n$, $n\times l$ costs $C_{\rm mm}mnl$ flops, and an economical QR factorization of an $m\times n$ dense matrix costs $C_{\rm qr}mn\min\{m,n\}$ flops, and the running time of {\sf randQB\underline{~}b} and {\sf randQB\underline{~}EI}  is about
\[
\begin{array}l
T_{{\sf randQB\underline{~}b} }\sim 3C_{\rm mm}mnr+C_{\rm mm}mr^2+{2\over s}C_{\rm qr}mr^2,\\
T_{{\sf randQB\underline{~}EI} }\sim 2C_{\rm mm}mnr+C_{\rm mm}(3m+n)r^2/2+{2\over s}C_{\rm qr}mr^2,
\end{array}
\]
where   $r=sb$ is the number of columns in the final $Q$. Because $r$ is usually much smaller than $m$ and $n$, the flop count of {\sf randQB\underline{~}EI} is about 2/3 of that of
{\sf randQB\underline{~}b}.

As mentioned in the previous section, the two variants of blocked randomized QB algorithm are restricted in practical applications   since the error indicator $E$   measured in Frobenius norm might fail  in some situations.

We aim to propose an algorithm  measuring the error using the spectral norm. We first explain the principle in generating the basis matrix of ${\cal R}_\theta(A)$ block by block.
Write the SVD factors of $A=U\Sigma V^T$ as
\begin{equation}
U=[U_{1}~~ U_{2}~\ldots~ U_{h}], \quad \Sigma={\rm diag}(\Sigma_{1},\ldots, \Sigma_{h}), \quad V=[V_{1}~~ V_{2}~\ldots~ V_{h}],\label{2.2}
\end{equation}
where $\{U_{j}\}_{j=1}^{h-1}$, $\{\Sigma_{j}\}_{j=1}^{h-1}$, $\{V_{j}\}_{j=1}^{h-1}$ have $b$ columns and the column sizes of $U_{h}, \Sigma_{h}, V_{h}$  depend on the dimension of the matrix $A$. Suppose that within the threshold $\theta$, the numerical range of $A$ can be characterized as
\begin{equation}
{\cal R}_\theta(A)={\rm span}\{U_{1}, U_{2},\ldots, U_{s}\}.\label{2.3}
\end{equation}
Without loss of generality, the numerical rank $k$ is assumed to  be a multiple of $b$, i.e., $k=sb$.
Clearly,
\begin{equation}
A=U_1\Sigma_{1}V_1^T+U_2\Sigma_{2}V_2^T+\cdots+U_s\Sigma_{s}V_s^T+\cdots+U_h\Sigma_{h}V_h^T,\label{2.4}
\end{equation}
and the matrix $A-U_1U_1^TA$ has the same set of singular value matrix  $\{\Sigma_{j}\}_{j=2}^h$ as those of $A$ except that its first largest $b$ singular values are replaced by 0. Thus, within the same threshold $\theta$,
\begin{equation}
{\cal R}_\theta(A-U_1U_1^TA)={\rm span}\{U_2,\ldots, U_s\},\quad {\rm rank}_\theta(A-U_1U_1^TA)={\rm rank}_\theta(A)-b.\label{2.5}
\end{equation}
Similarly, the matrix $A-U_1U_1^TA-U_2U_2^TA$ has $\{\Sigma_{j}\}_{j=3}^s$  and $2b$ zeros as its singular value set, and
\begin{equation}
\begin{array}{rl}
{\cal R}_\theta(A-U_1U_1^TA-U_2U_2^TA)&={\rm span}\{U_3,\ldots, U_s\},\\
{\rm rank}_\theta(A-U_1U_1^TA-U_2U_2^TA)&={\rm rank}_\theta(A-U_1U_1^TA)-b={\rm rank}_\theta(A)-2b.
\end{array}\label{2.6}
\end{equation}

In practical computations, $\{U_j\}_{j=1}^h$ are not known in advance.
To find the approximate numerical rank and numerical range space of $A$, we begin by finding a column-orthonormal matrix $Q_1\in {\mathbb R}^{m\times b}$ in the  range space of $A$ such that ${\cal R}(Q_1)\approx {\cal R}(U_1)$. This target can be accomplished efficiently by applying the randomized QB algorithm of $A$. One can also combine the power scheme to improve the  quality of column-orthonormal matrix $Q_1$. If ${\cal R}(Q_1)$ aligns well with ${\cal R}(U_1)$, the matrix  $(AQ_1)^T(AQ_1)$  has an eigenvalue matrix close to $\Sigma_{1}^2$.  Moreover, it can be shown that the numerical rank of $A-Q_1Q_1^TA$ is reduced by $b$ from that of $A$ (see the analysis in Theorem \ref{thm:3.6}).
Similarly, one can compute $Q_2$ in the numerical range of $A-Q_1Q_1^TA$, which is also in the numerical range of $A$. Here $Q_2$ and $Q_1$ are mutually column-orthonormal, since $Q_1$ is in the left kernel space of $A-Q_1Q_1^TA$.
The numerical rank of $A-Q_1Q_1^TA$ is reduced again.
 This process continues recursively and terminates when the numerical rank of $A$ is determined and a column-orthonormal basis for the numerical range is obtained.

In practical computations,  we combine the randomized QB algorithm with the power scheme to compute $Q_1, Q_2,\ldots$. However, a direct computation of the matrix $\hat A_1:=A-Q_1Q_1^TA$ is costly. Instead  we compute the sample matrix $Y_q=(\hat A_1\hat A_1^T)^q\hat A_1\Omega$ implicitly by employing a similar idea in \cite{llz}. To be precise, we choose
an $n\times b$ random Gaussian matrix, and initialize $Y_0$  as $Y_0=\hat A_1\Omega=(I-Q_1Q_1^T)(A\Omega)$, and then
the computation of $Y_q$ can be simplified as below
\[
\begin{array}{rl}
Y_1&=(\hat A_1\hat A_1^T)Y_0=(I-Q_1Q_1^T)AA^T(I-Q_1Q_1^T)Y_0\\
&=(I-Q_1Q_1^T)[AA^T(I-Q_1Q_1^T)]Y_0=(I-Q_1Q_1^T)[AA^T(I-Q_1Q_1^T)](A\Omega),\\
Y_2&=(\hat A_1\hat A_1^T)Y_1=(I-Q_1Q_1^T)AA^T(I-Q_1Q_1^T)Y_1\\
&=(I-Q_1Q_1^T)[AA^T(I-Q_1Q_1^T)][AA^T(I-Q_1Q_1^T)](A\Omega)\\
&=(I-Q_1Q_1^T)[AA^T(I-Q_1Q_1^T)]^2(A\Omega),\\
&\ldots\ldots\ldots\\
Y_q&=(I-Q_1Q_1^T)[AA^T(I-Q_1Q_1^T)]^q(A\Omega).
\end{array}
\]
This means the computation of $Y_q$  is equivalent to the power iteration on $AA^T(I-Q_1Q_1^T)$ combined with projection $I-Q_1Q_1^T$ on the final iterate,
in which the projection  $(I-Q_1Q_1^T)Z$ is computed $Z-Q_1(Q_1^TZ)$ instead for the sake of less flops.
The computational cost of explicit and implicit schemes is listed in Table  \ref{table1}.
   Usually $b$ is much smaller than $n$,  we see that the number of flops is reduced significantly in the implicit  formulation  of $Y_i$.

We describe the blocked low-rank approximation algorithm ({\sf blarank}) in Algorithm \ref{blarank}, in which
the approximate basis matrix $Q$ of ${\cal R}(A)$ is initialized as empty and then is incrementally built through $Q:=[Q~~Q_i]$.
The power iteration on $AA^T(I-Q_iQ_i^T)$ is implemented on the orthogonal QR factor.
This  will  increase the computational cost, but can maintain numerical stability \cite[Appendix]{gu}.
As stated in \cite{mv}, in floating point operations, a direct power scheme with big $q$ will  cause
substantial loss of accuracy  whenever the singular values of $A$ have a large dynamic range.
The applications of $AA^T$ on $(I-QQ^T)A\Omega_i$ and a small block size $b$ reduce the range of singular values,  which helps avoiding the substantial loss of accuracy for some big values of $q$.

\renewcommand\tabcolsep{30.0pt}
\begin{table}\label{table1}
\begin{center}
\caption{Comparison of computational cost of explicit and implicit formulation of $(AA^T)^qA\Omega$}
\begin{tabular}{*{28}{c}}
\hline
     &Explicit Scheme& Implicit Scheme\\\hline
$q=0$ &$3C_{\rm mm}mnb$&$~C_{\rm mm}mnb+2C_{\rm mm}mb^2$\\
$q=1$& $5C_{\rm mm}mnb$ & $3C_{\rm mm}mnb+4C_{\rm mm}mb^2$\\
$q=2$&$7C_{\rm mm}mnb$ & $5C_{\rm mm}mnb+6C_{\rm mm}mb^2$\\\hline
\end{tabular}
\end{center}
\end{table}

\begin{algorithm} \caption{({\sf blarank}) The blocked low-rank approximation  within given threshold} \label{blarank}

\begin{algorithmic}[1]
\REQUIRE $m\times n$ matrix $A$ with $m\ge n$,  block size $b$, threshold $\theta$ and power scheme parameter $q$.

\ENSURE column-orthonormal $m\times r$  basis matrix $Q$ of  the approximate numerical range,   and   rank-$r$ approximation $\tilde A_r=QQ^TA$.

\STATE Initialize $Q=[~]$.

\FOR { $i=1,2,3,\ldots$}

\STATE Set   $\Omega_i={\sf randn}(n,b)$, $Y_0={\sf orth}(A\Omega_i)$.

\FOR {$j=1,2,\ldots,q$}

\STATE Compute $Y_{j-1}:=Y_{j-1}-Q(Q^TY_{j-1})$.

\STATE   Compute $Z={\sf orth}(A^TY_{j-1})$.

\STATE Compute $Y_j={\sf orth}(AZ)$.

\ENDFOR

\STATE  Compute $\widehat  Q_i={\sf orth}(Y_q-Q(Q^TY_q))$.\qquad

\STATE Compute $[\widehat U_i,\widehat \Lambda_i]={\sf eig_\downarrow}(\widehat  Q_i^TAA^T\widehat  Q_i)$ { and $\widehat Q_i:=\widehat  Q_i\widehat U_i$}.

\IF {minimum $t$ is found so that $\widehat \Lambda_i(t,t)<\theta^2$}

\STATE  Set  $\widehat Q_i(:,t:b)=[~]$ and $r=(i-1)b+t-1$.

\STATE   Goto step 15, and then break the $i$-loop.

\ENDIF

\STATE Compute  $ Q_i={\sf orth}\big((I-QQ^T){\widehat  Q_i}\big)$  and set $Q=[Q \quad   Q_i]$.

\ENDFOR

\end{algorithmic}

\end{algorithm}

As to the complexity analysis, we note that in {\sf blarank}, it costs $C_{\rm mm}mnb+C_{\rm qr}{mb^2}$
 in step 3, $2C_{\rm mm}m(i-1)b^2$ in step 5, $C_{\rm mm}mnb+C_{\rm qr} mb^2$ in step 6 and $C_{\rm mm}mnb+C_{\rm qr} nb^2$ step 7,    $2C_{\rm mm}m(i-1)b^2+C_{\rm qr} mb^2$ in step 9, and
 $C_{\rm mm}(mnb+nb^2)+C_{\rm eig}b^3{+C_{\rm mm}mb^2}$ in step 10. Updating $Q_i$ in step 15 needs $2C_{\rm mm}m(i-1)b^2 +C_{\rm qr}mb^2$. If the approximate numerical rank is $r=sb$, then
 the total cost is about
 \[
 \begin{array}{rl}
 T_{\sf blarank}&\sim 2C_{\rm mm}mnbs+\sum\limits_{i=1}^s\big(2C_{\rm mm}m(i-1)b^2(q+2)\big)\\
 &\qquad\qquad\qquad + sq\big(C_{\rm qr} mb^2+2C_{\rm mm}mnb\big)+3C_{\rm qr} mb^2s+C_{\rm eig}b^3s+C_{\rm mm}(mb^2s+nb^2s)\\
 &\sim C_{\rm mm}\Big(m(q+2)r^2+(2q+2)mnr+(m+n)r^2/s\Big)+(2q+3)C_{\rm qr}mr^2/s+C_{\rm eig} r^3/s^2,
 \end{array}
 \]
 in which $1\le s\le r$. When $q=0$, the cost is comparable to the one for {\sf randQB\_EI};
and when $s=r, b=1$, the algorithm generates a single vector stored in $Q_i$ in one iteration, which is similar to the  {\sf larank} algorithm \cite{llz}.
In this case it costs the least operations, however, it does not necessarily mean {\sf larank} is the most efficient, since

{\it Flop counting is a necessarily  crude approach to the measurement of program
efficiency and it ignores subscripting, memory traffic, and other overheads associated
with program execution \cite[\S 1.2.4]{gv2}}.

The blocked algorithm {\sf blarank} uses matrix-matrix multiplications and highly optimized routines for standard matrix factorizations that need less data movement and less communications. This results in a potentially efficient algorithm, whose  efficiency  will be displayed through numerous experiments in section 4 .

To obtain a numerical range space with high accuracy,  orthogonalizing $Q_i$ against $Q$
is not sufficient to maintain stability. We suggest replacing ``{\sf orth}'' in step 9 of {\sf blarank} with ``{\sf reorth2}'', as given in Algorithm \ref{sblarank}.
 We refer to the algorithm as the {\sf sblarank} algorithm.
  It is mathematically equivalent   with {\sf blarank}, but behaves more stable  in the later numerical experiments. The twice-orthogonalization operator {\sf reorth2}
  is used to make $Q_i$ to be highly orthogonal to previously generated factors $Q_1,\ldots, Q_{i-1}$ and   align ${\cal R}(Q_i)$ well with ${\cal R}(U_i)$  as much as possible.

\begin{algorithm}\caption{({\sf sblarank}) The stabilized blocked low-rank approximation within given threshold}  \label{sblarank}
\begin{algorithmic}[1]
\REQUIRE $m\times n$ matrix $A$ with $m\ge n$,  block size $b$, threshold $\theta$ and power scheme parameter $q$.

\ENSURE column-orthonormal $m\times r$  basis matrix $Q$ of  the approximate numerical range,   and   rank-$r$ approximation $\tilde A_r=QQ^TA$.

\STATE Initialize $Q=[~]$.

\FOR { $i=1,2,3,\ldots$}

\STATE Set $\Omega_i={\sf randn}(n,b)$, $Y_0={\sf orth}(A\Omega_i)$.

\FOR {$j=1,2,\ldots,q$}

\STATE Compute $Y_{j-1}:=Y_{j-1}-Q(Q^TY_{j-1})$.

\STATE   Compute $Z={\sf orth}(A^TY_{j-1})$.

\STATE Compute $Y_j={\sf orth}(AZ)$.

\ENDFOR

\STATE  Compute $\widehat Q_i={\sf reorth2}(Y_q-Q(Q^TY_q))$.\qquad

\STATE Compute $[\widehat U_i,\widehat\Lambda_i]={\sf eig_\downarrow}(\widehat Q_i^TAA^T\widehat Q_i)$ {  and $\tilde Q_i:=\widehat  Q_i\widehat U_i$}.

\IF {minimum $t$ is found so that  $\widehat\Lambda(t,t)<\theta^2$}

\STATE  Set  $\widehat  Q_i(:,t:b)=[~]$ and  $r=(i-1)b+t-1$.

\STATE   Goto step 15, and then break the $i$-loop.

\ENDIF

\STATE Compute $Q_i:={\sf orth}((I-QQ^T){ \tilde Q_i})$  and set $Q=[Q \quad Q_i]$.

\ENDFOR

\end{algorithmic}

\end{algorithm}

\begin{remark} In Algorithm  \ref{blarank} and Algorithm \ref{sblarank}, the computation of final $Q_1$ is exactly the orthogonal iteration with Rayleigh-Ritz acceleration \cite[Chapter 8.3.7]{gv2}, which aims to find better approximate left singular subspace and singular value estimation. We do not use the oversampling technique to produce a higher quality of $Y_0$ in that
we aim to determine the numerical rank and the basis for ${\cal R}_\theta(A)$. The issue of a higher accuracy of singular subspace is not our focus. A provable result  in section 3 shows that a sight mismatch of ${\cal R}(Q_1)$  with the singular subspace ${\cal R}(U_1)$  will not influence the determination of the numerical rank.
\end{remark}

\section{Approximate subspace and numerical rank decrease}
As stated in \eqref{2.2}-\eqref{2.5}, the quality of column-orthonormal matrix $Q$ plays crucial roles in the accuracy of the algorithm.
The accuracy of the computed range space depends on how well the space ${\cal R}(Q)$ aligns
with the space  ${\cal R}(U_{(k)})$ spanned by   the $k$ dominant left singular vectors of $A$. If these two spaces were to match exactly,
then the algorithm would achieve perfectly optimal accuracy.

In this section we build some vital results to shed a light on the factors affecting the quality of the basis matrix $Q$, and prove that even if
${\cal R}(Q_j)$ does not align  ${\cal R}(U_j)$ very well, then under some mild assumptions,  the algorithm can still determine the right numerical rank and an approximate numerical range space with high quality.

To see this, we first give some preliminary results for later analysis.

%


\begin{lemma}{\rm(\cite[Theorem 5.8]{gu})}\label{lem:3.3} Let $\Omega=\Big[{\Omega_1\atop \Omega_2}\Big]$ be an $n\times \ell$ Gaussian matrix with $\Omega_1\in {\mathbb R}^{r\times \ell}$ of full row rank.
Then for a given $0\le \delta\ll 1$,  $\|\Omega_2\|_2\| \Omega_1^\dag\|_2\le {\cal C}_{\delta,r,\ell}$ with the probability at least $1-\delta$, where
\[
{\cal C}_{\delta,r,\ell}={2e\sqrt{\ell}\over \ell-r+1 }\Big({2\over \delta}\Big)^{1\over \ell-r+1}\left(\sqrt{n-r}+\sqrt{r}+\sqrt{2\log{2\over \delta}}\right).
\]

\end{lemma}

The results shows that when $r<\ell$, $\|\Omega_2\Omega_1^\dag\|_2$ might give a smaller bound than that for $r=\ell$.
 The theorem below displays the matrix and singular values approximation error,  the distance of the spaces spanned by   $U_{(t)}$  and $Q_1$, when the oversampling is not used in the proposed algorithm.

\begin{theorem}\label{thm:3.1} Let the SVD of an $m\times n$ matrix $A$  satisfy \eqref{1.0}-\eqref{1.2} and  $U_{(t)}$  contain the first $t$ {($t\ge b$)} columns of $U$. Given an $n\times b$ Gaussian matrix  $\Omega_1$,  let $Q_1=\widehat Q_1\widehat U_1$ and $\widehat Q_1={\sf orth}((AA^T)^qA\Omega_1)$ be generated from Algorithm \ref{blarank}, where
\[
\widehat \Omega:=V^T\Omega_1=\begin{array}l\left[\begin{array}{cc}
\widehat \Omega_{1}\\\widehat  \Omega_{2}\end{array}\right]
\end{array}
\begin{array}l b\\ n-b\end{array}
\]
and  $\widehat \Omega_{1}\in {\mathbb R}^{b\times b}$ is assumed to have full-rank. Then  $b\le t\le k$,
\begin{eqnarray}
&{}&\|(I-Q_1Q_1^T)A\|_{2,F}\le (1+\hat\tau_b^{4q}\|\widehat \Omega_{2}\widehat \Omega_{1}^{-1}\|_2^2)^{1/2}\|\Sigma_{(b)\bot}\|_{2,F}, \label{000}\\
&{}&\sigma_j\ge \sigma_{j}(Q_1^TA)\geq\frac{\displaystyle\sigma_j}{\displaystyle\sqrt{1+\hat \tau_j^{4q+2}\|\widehat\Omega_{2}\widehat\Omega_{1}^{-1}\|_{2}^{2}}}, \label{0000}\\
&{}&{\bf d}({\cal R}(Q_1),{\cal R}(U_{(t)}))\le \Big({\sigma_{t+1}\over \sigma_b}\Big)^{2q+1}\|\widehat \Omega_{2}\widehat \Omega_{1}^{-1}\|_2,\label{3.1}
\end{eqnarray}
where  $\hat \tau_j=\sigma_{b+1}/\sigma_j<1$ for $1\le j\le b$, and with $0<\delta<1$ and the probability at least $1-\delta$,
$\|\widehat \Omega_{2}\|_2\|\widehat \Omega_{1}^{-1}\|_2\le {\cal C}_{\delta,b,b}$, where ${\cal C}_{\delta,b,b}$  is defined in Lemma  \ref{lem:3.3}.
\end{theorem}

\begin{proof} \eqref{000} and \eqref{0000} are straightforward from Theorem \ref{thm:2.1}. For the estimate in \eqref{3.1}, we note by \eqref{dist} and \eqref{1.0} that
\begin{equation}
{\bf d}({\cal R}(Q_1),{\cal R}(U_{(t)}))=\|U_{(t)\bot}^TQ_1\|_2=\|U_{(t)\bot}^TQ_1Q_1^TU_{(t)\bot}\|_2^{1/2},\label{eq1}
\end{equation}
where $Q=[Q_1\quad Q_{1\bot}]$ is an $m\times m$ orthogonal matrix. Observe $\widehat U_1$ is an orthogonal matrix, and hence
\[
{\cal R}(Q_1)={\cal R}(\widehat Q_1)={\cal R}((AA^T)^qA\Omega_1)={\cal R}(U(\Sigma\Sigma^T)^q\Sigma\widehat \Omega_1)={\cal R}(Z),
\]
where
\begin{equation}
Z=U(\Sigma\Sigma^T)^q\Sigma\widehat \Omega\widehat \Omega_{1}^{-1}\Big(\Sigma_{(b)}^{2q+1}\Big)^{-1}=U\left[\begin{array}c I_b\\ F\end{array}\right],\quad F= (\Sigma_{(b)\bot}\Sigma_{(b)\bot}^T)^{q}\Sigma_{(b)\bot}\widehat \Omega_{2}\widehat \Omega_{1}^{-1}\Big(\Sigma_{(b)}^{2q+1}\Big)^{-1}.\label{e0}
\end{equation}

Notice that the orthogonal projector
\begin{equation}
Q_1Q_1^T=ZZ^\dag=U\left[\begin{array}c I_b\\ F\end{array}\right](I_b+F^TF)^{-1}[I_b\quad F^T]U^T.\label{e1}
\end{equation}
Set $t=b+p$ and  $\hat F=F(p+1:m-b,:)$.
Substituting above relation into \eqref{eq1},   we get
\begin{eqnarray}
{\bf d}({\cal R}(Q_1),{\cal R}(U_{(t)}))&
=&\|\hat F(I_b+F^TF)^{-1}\hat F^T\|_2^{1/2}\nonumber\\
&\le&  {\|\hat F\|_2}\le  \Big({\sigma_{t+1}/\sigma_b}\Big)^{2q+1}\|\widehat \Omega_{2}\widehat \Omega_{1}^{-1}\|_2,\nonumber
\end{eqnarray}
where the distribution of the Gaussian matrix is rotationally invariant, and hence $\|\widehat \Omega_{2}\widehat \Omega_{1}^{-1}\|_2\le {\cal C}_{\delta,b,b}$.
\end{proof}


The result shows that the large singular values of $Q_1^TA$ can be computed more accurately due to a smaller singular value ratio $\hat\tau_j$.
 The singular value gap ratio $\sigma_{b+1}/\sigma_b$, the power scheme parameter $q$ and $\|\widehat \Omega_{2}\|_2\|\widehat \Omega_{1}^{-1}\|_2$ are key factors governing the distance of
subspaces ${\cal R}(U_{(b)})$ and ${\cal R}(Q_1)$. The distance  ${\bf d}({\cal R}(U_{(b)}), {\cal R}(Q_1))$ might not be  very small due to
a possible large factor  $\|\widehat \Omega_{2}\|_2\|\widehat \Omega_{1}^{-1}\|_2$. However, if we take ${\cal R}(U_{(t)})$ with $t=k$, i.e., the numerical range space ${\cal R}_\theta(A)$, then the distance might be small because of a possible smaller factor $\sigma_{k+1}/\sigma_b$. This means that even ${\cal R}(Q_1)$ does not align well with ${\cal R}(U_{(b)})$, $Q_1$  might be still a good basis of  ${\cal R}_\theta(A)$. That is why we use non-oversampling strategy in the randomized algorithm for computing $Q_1$.
 In addition, when an unfortunate choice of $\Omega$ causes
    a bad mismatch of ${\cal R}(Q_1)$ with   ${\cal R}(U_{(b)})$, it means
some columns in $Q_1$  might capture  a fraction of the range spanned by the singular vectors associated with the tail singular values.
In the forthcoming theorem, we will prove that if this fraction  can be controlled below a certain accuracy, then it will not affect the algorithm to determine
the numerical rank and numerical range of $A$.
To see this, we first recall a lemma in \cite[Theorem 4.2]{llz}.

\begin{lemma}[\cite{llz},Theorem 4.2]\label{lem:3.2} Let $A\in {\mathbb R}^{m\times n}$ assert the SVD: $A=U\Sigma V^T$ whose singular values satisfy (\ref{1.2}). Given an $m\times j$ column-orthonormal matrix $W$ whose column vectors belong to ${\cal R}_\theta(A)$. Then the singular values of $A-WW^TA$ can be indexed as $\sigma_1',\sigma_2',\ldots,\sigma_n'$ such that
\[
\begin{array}l
\sigma_1\ge \sigma_1'\ge \cdots \ge \sigma_{k-j}'\ge \sigma_k,\quad  \sigma_{k-j+1}'=\sigma_{k-b+2}'=\cdots=\sigma_{k}'=0,\\
and \quad \sigma_i'=\sigma_i\quad for\quad i=k+1,\ldots,n.
\end{array}
\]
\end{lemma}

Lemma \ref{lem:3.2} says that if ${\cal R}(W)\subseteq {\cal R}_\theta(A)$ with $W^TW=I_j$, then the deflated matrix $(I-WW^T)A$  has  $k-j$ singular values greater than $\theta$, and its largest singular value never exceeds $\sigma_1$, the largest singular value of $A$. Moreover, the singular values of $A$ that are below $\theta$ are also the singular values of $(I-WW^T)A$, and the numerical rank gap ratio of $(I-WW^T)A$ is $\sigma_{k-j}'/\sigma_{k+1}$, and this gap ratio is always greater than the initial numerical gap  ratio $\sigma_{k}/\sigma_{k+1}$ of $A$, and enables the deflation process to be  effective for the rank-revealing algorithm.

 With Lemma \ref{lem:3.2}, in the theorem below, we prove that when a column-orthonormal matrix $Q_1$ captures a small fraction out side of ${\cal R}_\theta(A)$, it still maintains the singular values of $A$ below $\theta$ with some certain accuracy.

\begin{theorem} \label{thm:3.6}  With the notation of Theorem \ref{thm:3.1}, suppose that
\[\epsilon={\bf d}( {\cal R}(Q_1),{\cal R}_\theta(A)) <1 \mbox{ and } \sigma_k-\epsilon\|A\|_2>\theta>\sigma_{k+1}+\epsilon \|A\|_2,\]
 then $B=(I_m-Q_1Q_1^T)A$ maintains the singular values of $A$ below $\theta$ with absolute error not exceeding $\epsilon\|A\|_2$. Moreover,
\[{\rm rank}_\theta(B)={\rm rank}_\theta(A)-b.\]
\end{theorem}

\begin{proof} Decompose $Q_1$ into the form
\begin{equation}
Q_1=\dot{Q}_1+\ddot{ Q}_1=\check{Q}_1\check{R}+\hat Q_1 \hat R=[{\check Q}_1~~{\hat Q}_1]X,\quad X=\Big[{{\check R}\atop {\hat R}}\Big],\label{e7}
\end{equation}
where $\dot{Q}_1\in {\cal R}_\theta(A)$, $\ddot{Q}_1\in  {\cal K}_\theta(A^T)$, and the column-orthonormal matrices ${\check Q}_1$, $\hat{Q}_1$ are obtained from the thin  QR of
$\dot{Q}_1$ and $\ddot{Q}_1$, respectively. Here $\|\ddot{Q}_1\|_2=\epsilon$ since by \eqref{eq1},
\[\|\ddot{Q}_1\|_2=\|U_{(k)\bot}U_{(k)\bot}^TQ_1\|_2={\bf d}({\cal R}(Q_1), {\cal R}_\theta(A))\le\Big({\sigma_{k+1}\over \sigma_b}\Big)^{2q+1}\|\widehat\Omega_2\widehat\Omega_1^{-1}\|_2.\]

In \eqref{e7}, the column size of $\check Q_1$ is crucial for our analysis.
 By the properties of the singular values, it is obvious
that $|\sigma_i(Q_1)-\sigma_i(\dot Q_1)|\le \|\ddot Q_1\|_2=\epsilon$ for $i=1,2,\ldots,b$. Thus $\sigma_b(\dot Q_1)\ge 1-\epsilon$ and the column-orthonormal matrix $\check Q$ has $b$ column vectors. The column size of $\hat Q_1$  depends on the properties of $\ddot Q_1$, and we denote it as  $s$.  Clearly,  $\hat R$ is an upper trapezoidal  matrix  of size $s\times b$.

It is obvious  that  the relation $Q_1^TQ_1=I_b$ implies $X^TX=I_b$.  By setting   $Y=\big[{I_b\atop 0_{s\times b}}\big]$, we get from \eqref{e7} and \eqref{dist} that
\[
\begin{array}{rl}
{\bf d}({\cal R}(Q_1), {\cal R}({\check Q}_1))&=\|Q_1Q_1^T-{\check Q}_1{\check Q}_1^T\|_2=\|[{\check Q}_1~{\hat Q}_1](XX^T-YY^T)[{\check Q}_1~{\hat Q}_1]^T\|_2\\
 &=\|XX^T-YY^T\|_2=\|X^T\big[{0_{b\times s}\atop I_s}\big]\|_2=\|\hat R\|_2=\|\ddot Q_1\|_2{=\epsilon}.
 \end{array}
\]

Let    $\check B=(I_m-\check Q_1\check Q_1^T)A$. Then
\begin{equation}
\|\check B-B\|_2=\|(Q_1Q_1^T-{\check Q}_1{\check Q}_1^T)A\|_2\le \epsilon \|A\|_2.\label{e8}
\end{equation}
By applying Lemma \ref{lem:3.2}, we know that $\check B$ has $b$ zero singular values, and $k-b$ singular values between $\sigma_1$ and $\sigma_k$, and $n-k$ singular values that are identical to the smallest $n-k$ singular values of $A$. Reindexing these singular values $\{\check \sigma_i\}$ as
\[
\sigma_1\ge \check \sigma_1\ge\cdots\ge  \check \sigma_{k-b}\ge\sigma_k>\theta>\sigma_{k+1}\ge\check \sigma_{k-j+1}\ge  \cdots\ge \check\sigma_{n-b+1}=\cdots=\check\sigma_{n}=0,
\]
where $\{\check \sigma_j\}_{j=k-b+1}^n$ consists of  $n-k$ singular values of $A$ below $\theta$ and $b$ zero values.

 If we denote by $\{\tilde \sigma_i\}$ the singular values of $B$, then by \eqref{e8} and the perturbation theory for singular values,
$|\tilde\sigma_i-\check\sigma_i|\le \|B-\check B\|_2\le \epsilon \|A\|_2$ for $i=1,2,\ldots,n$. This proves that $B$ maintains the singular values of $A$ below $\theta$ with absolute error not exceeding $\epsilon \|A\|_2$.

For the singular values of $B$, we observe that
$\tilde\sigma_{k-b}$, $\tilde\sigma_{k-b+1}$  lie in the intervals as
 \begin{equation}
 \begin{array}l
\tilde\sigma_{k-b}\in  [\check\sigma_{k-b}-\epsilon \|A\|_2, ~\check \sigma_{k-b}+\epsilon \|A\|_2],\quad \mbox{with}\quad \check\sigma_{k-b}\ge \sigma_k,\\
\tilde\sigma_{k-b+1}\in  [\check\sigma_{k-b+1}-\epsilon \|A\|_2,~ \check \sigma_{k-b+1}+\epsilon \|A\|_2],\quad \mbox{with}\quad \check\sigma_{k-b+1}\le \sigma_{k+1},
\end{array}\label{3.5}
 \end{equation}
 from which we derive that
\[
 \tilde\sigma_{k-b}\ge \sigma_k-\epsilon\|A\|_2>\theta> \sigma_{k+1}+\epsilon \|A\|_2\ge \check \sigma_{k-b+1}+\epsilon \|A\|_2\ge \tilde\sigma_{k-b+1}.
\]
This implies that   ${\rm rank}_\theta (B)=k-b={\rm rank}_\theta (A)-b$, and the proof is complete.
\end{proof}

With a similar technique, we get the following result.

\begin{corollary}\label{cor:3.8} Let the SVD of the $m\times n$ matrix $A$  satisfy \eqref{1.0}-\eqref{1.2}, and  $Q_{[\ell]}=[Q_1~Q_2~\ldots~ Q_\ell]$ with $Q_i\in {\mathbb R}^{m\times b}$ being computed via the {\sf blarank} or {\sf sblarank} algorithm.
If the orthogonal projection $\ddot Q_{[\ell]}$ of $Q_{[\ell]}$ on the numerical left kernel space ${\cal K}_\theta(A^T)$ of $A$ satisfies $\|\ddot Q_{[\ell]}\|_2={\bf d}({\cal R}(Q_{[\ell]}), {\cal R}_\theta(A))=\epsilon<1$ and
$\sigma_{k+1}+\epsilon\|A\|_2<\theta<\sigma_k-\epsilon\|A\|_2$, then
\[
{\rm rank}_\theta({\cal P}_\ell^\bot A)={\rm rank}_\theta(A)-b\ell,
\]
where ${\cal P}_\ell^\bot=I-Q_{[\ell]}Q_{[\ell]}^T$ and ${\cal P}_\ell^\bot A$ maintains the singular values of $A$ below $\theta$ with absolute error not exceeding $\epsilon\|A\|_2$.
Moreover, let $U_{[\ell]}=[U_1~U_2~\ldots~ U_\ell]$ with $U_i$ being defined in \eqref{2.2}, and if
\[
\epsilon_\ell={\bf d}({\cal R}(Q_{[\ell]}),{\cal R}(U_{[\ell]}))<1, \mbox{ and } \sigma_{b\ell}-\sigma_{b\ell+1}>2\epsilon_{\ell}\|A\|_2,
\]
  then
${\cal P}_\ell^\bot A$ maintains  the singular values of $A$ below $\sigma_{b\ell}$; and also those of $\tilde {\cal P}_\ell^\bot A$ with absolute error not exceeding $\epsilon_{\ell}\|A\|_2$, in which $\tilde{\cal P}_\ell^\bot=I-U_{[\ell]}U_{[\ell]}^T$.
\end{corollary}

The following theorem provides the bound for $\|{\cal P}_\ell^\bot A\|_2$ and the estimate of $\sigma_j(Q_\ell^TA)$.

\begin{theorem}\label{thm:3.7} With the notation of  Corollary \ref{cor:3.8}, let
$\Omega_\ell$ be  an $n\times b$ Gaussian matrix,   and $Q_\ell={\sf orth}((BB^T)^qB\Omega_\ell)$, where
$B={\cal P}_{\ell-1}^\bot A$ is assumed to have the SVD as $B=\widetilde U\widetilde\Sigma \widetilde V^T$ with  $\widetilde\sigma_{i}$ being its $i$th largest singular value.
Denote
  \[
 \widetilde \Omega:=\widetilde V^T\Omega_{\ell}=\left[\begin{array}c \widetilde\Omega_{1}\\ \widetilde\Omega_{2}\end{array}\right]\begin{array}l b\\ n-b\end{array}
 \]
 If
\[
\epsilon_{\ell-1}={\bf d}({\cal R}(Q_{[\ell-1]}),{\cal R}(U_{[\ell-1]}))<1,\quad \mbox{and~~}  \sigma_{b\ell}-\sigma_{b\ell+1}>2\epsilon_{\ell-1}\|A\|_2,
\]
then for $0<\delta<1$,
\[
\|{\cal P}_{\ell}^\bot A\|_{2,F}\le (1+\tilde\tau_b^{4q}\|\widetilde \Omega_{2}\widetilde \Omega_{1}^{-1}\|_2^2)^{1/2}\|{ \widetilde\Sigma_{(b)\bot}}\|_{2,F}\le
(1+\tilde\tau_b^{4q}{\cal C}_{\delta,b,b}^2)^{1/2}\|{\widetilde\Sigma_{(b)\bot}}\|_{2,F},
\]
holds with the probability at least $1-\delta$, where   ${\cal C}_{\delta,b,b}$ is
 defined in Lemma \ref{lem:3.3}, and
  \begin{equation}
  \tilde\tau_b={\widetilde\sigma_{b+1}\over\widetilde\sigma_{b}}\le {\sigma_{b\ell+1}+\epsilon_{\ell-1} \|A\|_2\over \sigma_{b\ell}-\epsilon_{\ell-1} \|A\|_2}<1,\label{eq6}
  \end{equation}
with
\begin{equation}
\begin{array}l
|\widetilde \sigma_{j}-\sigma_{b(\ell-1)+j}|\le \epsilon_{\ell-1} \|A\|_2,\quad j=1,2,\ldots, n-b(\ell-1),\\
\widetilde \sigma_{j}\le \epsilon_{\ell-1} \|A\|_2,\quad j=n-b(\ell-1)+1,\ldots, n.
\end{array}
\label{sv}
\end{equation}
 Moreover, the singular values $\sigma_{j}(Q_{\ell}^TA)$ estimated via {\sf blarank} and {\sf sblarank} satisfy
 \begin{equation}
\widetilde\sigma_{j}\ge \sigma_{j}(Q_{\ell}^TA)\geq\frac{\widetilde\sigma_j}{\sqrt{1+\|\widetilde\Omega_{2}\|_{2}^{2}
\|\widetilde\Omega_{1}^{-1}\|_{2}^{2}\left(\frac{\widetilde\sigma_{b+1}}{\widetilde\sigma_{j}}\right)^{4q+2}}}.\label{eq4}
\end{equation}
\end{theorem}

\begin{proof}
By Theorem \ref{thm:3.1}, one can get
\begin{equation}
\|(I-Q_\ell Q_\ell^T)B\|_{2,F}\le  (1+\tilde\tau_b^{4q}\|\widetilde \Omega_{2}\widetilde \Omega_{1}^{-1}\|_2^2)^{1/2}\|{ \widetilde\Sigma_{(b)\bot}}\|_{2,F}\le
(1+\tilde\tau_b^{4q}{\cal C}_{\delta,b,b}^2)^{1/2}\|{\widetilde\Sigma_{(b)\bot}}\|_{2,F},\label{eq2}
\end{equation}
where
\begin{equation}
(I-Q_\ell Q_\ell^T)B=(I-Q_\ell Q_\ell^T){\cal P}_{\ell-1}^\bot A=(I-Q_{[\ell]} Q_{[\ell]}^T)A={\cal P}_{\ell}^\bot A.\label{eq5}
\end{equation}
This yields the estimates for $\|{\cal P}_{\ell}^\bot A\|_{2,F}$.

For the estimate of $\tilde \tau_b$, {note  Lemma \ref{lem:3.2} ensures}
that $\tilde {\cal P}_{\ell-1}^\bot A$ has  $\sigma_{b(\ell-1)+1},$ $\ldots, \sigma_{n}$ and other $b(\ell-1)$ zeros as its singular values, among which $\sigma_{b\ell+1}$ is the $(b+1)$th largest singular value.
With Corollary \ref{cor:3.8} and the perturbation theory of singular values, we get
\[
|\tilde\sigma_j-\sigma_{b(\ell-1)+j}|\le \|{\cal P}_{\ell-1}^\bot A-\tilde {\cal P}_{\ell-1}^\bot A\|_2=\|(Q_{[\ell-1]}Q_{[\ell-1]}^T-U_{[\ell-1]}U_{[\ell-1]}^T)A\|_2\le \epsilon_{\ell-1}\|A\|_2,
\]
where $\sigma_{b(\ell-1)+j}=0$ for $j=n-b(\ell-1)+1,\ldots,n$. This yields \eqref{sv}, and the relation in \eqref{eq6}.

 As for the estimate of $\sigma_j(Q_\ell^TA)$, note that $Q_\ell={\sf orth}((BB^T)^qB\Omega_\ell)$, it follows from  Theorem  \ref{thm:2.1} that
 \[
 \widetilde\sigma_{j}\ge \sigma_{j}(Q_{\ell}^TB)\geq\frac{\widetilde\sigma_j}{\sqrt{1+\|\widetilde\Omega_{2}\|_{2}^{2}
\|\widetilde\Omega_{1}^{-1}\|_{2}^{2}\left(\frac{\widetilde\sigma_{b+1}}{\widetilde\sigma_{j}}\right)^{4q+2}}},
\]
where $Q_{\ell}^TB=Q_\ell^TA$.  This gives the estimate in \eqref{eq4}.
\end{proof}


%
\begin{remark} The results in Theorem \ref{thm:3.7} and Corollaries \ref{cor:3.8}   show that if ${\cal R}(Q_{[\ell-1]})$ aligns ${\cal R}(U_{[\ell-1]})$ very well,
$\|(I-QQ^T)A\|_2$ approximates $\sigma_{k+1}$ up to some constants depending on $(1+\tilde\tau_b^{4q}{\cal C}_{\delta,b,b}^2)^{1/2}$ from optimal,
and the singular values of $Q_\ell^TA$ (or eigenvalues of $Q_\ell^TAA^TQ_\ell$) approximates well those of  $\Sigma_\ell$ (or $\Sigma_\ell^2$) with certain errors depending on $\epsilon_{\ell-1}\|A\|_2$ and singular value decaying rate.

Regarding the magnitude of $\epsilon_{l}$, we note that when $q=0$, $Q_{[l]}$ is essentially the column-orthogonal matrix obtained by orthogonalizing  $Y=A[\Omega_1,\ldots, \Omega_{l}]$ (denoted as $A\Omega_{[l]}$) with   block Modified Gram-Schmidt process. 
By a similar argument to that used in proving   \eqref{3.1}, we have
\[
{\bf d}({\cal R}(Q_{[l]}),{\cal R}(U_{[l]}))\le {\sigma_{bl+1}\over \sigma_{bl}}\|\widehat \Omega_{[l],2}\widehat \Omega_{[l],1}^{-1}\|_2,
\]
where $\widehat\Omega_{[l],1}, \widehat\Omega_{[l],2}$ are the first $bl$ and last $(n-bl)$ rows of $V^T\Omega_{[l]}$, respectively. A small gap ratio incurs 
a small magnitude for $\epsilon_l$. We believe the increase of $q$ will further decrease the magnitude for $\epsilon_l$.

\end{remark}

\section{Numerical experiments}
In this section, we use the {\sf blarank} and  {\sf sblarank} algorithms to calculate the numerical rank, numerical range and matrix approximation error of a matrix with a low numerical rank.
Existing {\sf larank} \cite{llz} and Lanczos-based PROPACK \cite{la} are compared for the same purpose.
Since the PROPACK can not automatically determine the singular values above the threshold,
 we preassign a target rank guess $s_0$ and compute  the $s_0$ largest singular values and corresponding
singular vectors.  If some of the computed singular values are already smaller
than $\theta$,  then $s_0$ is the right choice. Otherwise, increase $s_0$ by a predefined integer
$h$ repeatedly until some of the singular values fall below $\theta$.
 The increment of $s_k$ works very well in practice, however reruning Lanczos iterations with an increment $h$ is not economical in terms of the computational cost,
since some of the singular values have already been computed; and only a few new singular values
($h$ of them) need to be numerically evaluated.  In our experiments, we choose appropriate $s_0,h$ to ensure that the algorithm only needs  one update at most.

All the computations are carried out in Matlab R2012b on  a person HP notebook with an Intel Core i5-5200 CPU of 2.2GHz, 4GB of memory and machine precision $2.2\times 10^{-16}$.

\subsection{Synthetic Data}
We use a similar idea in \cite{llz} to construct the synthetic data.  The test matrix $A$ is of $2n\times n$ in the form $A=U\Sigma V^T$. The orthogonal matrices $U, V$ are randomly generated with appropriate sizes,
and the singular values in the diagonal of $\Sigma$ are distributed geometrically as
\[
\begin{array}c
1=\sigma_1\ge \cdots \ge \sigma_{k_1}=10^{-4}, \qquad 10^{-6}=\sigma_{k_1+1}\ge \cdots \ge \sigma_{k_2}=10^{-8},\\
  10^{-10}=\sigma_{k_2+1}\ge \cdots\ge \sigma_n=10^{-15}.
  \end{array}
\]
We generate two types of matrices:

I) $n=400, k_1=10, k_2=20$ and $\theta=10^{-5}$. The true numerical rank $k=10$.;

II)  $n=800, k_1=5, k_2=20$ and $\theta=10^{-9}$. The true numerical rank $k=20$.

\renewcommand\tabcolsep{3.5pt}
\begin{table}[!htbp]\label{table2}
	\centering
	\caption{The numerical results for Type I matrices}
	\begin{tabular}{c|cc|c|cc|cc}
			\midrule
		{Type I }&\multicolumn{2}{c|}{\sf lansvd}&{\sf larank }& \multicolumn{2}{|c|} {\sf blarank}&\multicolumn{2}{c} {\sf sblarank}  \\\cline{1-8}
$q=1$&$s_0=5$&$s_0=10$&&$b=5$&$b=10$&$b=5$&$b=10$\\\hline
time(seconds)&0.063&0.031&0.064&0.015 &0.011&  0.017& 0.015\\
crank($r$) &10&10&10&10&10&10&10\\
 $\|I-Q^TQ\|_2$&    1.04e-11 &  6.60e-13 &  1.78e-15 &  2.23e-15 &  1.70e-15 &  1.49e-15  & 1.57e-15\\
   range error& 3.84e-13 &  3.96e-13 &  3.72e-13 &  1.01e-6 &  2.67e-6 &  3.45e-6 &  1.63e-6\\
   $\|\tilde A_r-A_k\|_2$&1.00e-5&   1.00e-5 &  1.07e-15 &  1.02e-10 &  2.72e-10 &  3.46e-10&  1.66e-10\\\hline\hline
$q=2$&$s_0=5$&$s_0=10$&&$b=5$&$b=10$&$b=5$&$b=10$\\\hline
time(seconds)&0.078&0.062&0.078&0.016&          0.014 & 0.016&0.015\\
crank($r$) &10&10&10&10&10&10&10\\
 $\|I-Q^TQ\|_2$& 4.28e-12 &  1.69e-11 &  1.12e-15 &  2.27e-15 &  1.15e-15 &  2.49e-15 &  2.23e-15\\
   range error& 3.95e-13 &  4.01e-13 &  3.98e-13  & 2.67e-10 &  3.55e-11 &  1.31e-10 &  9.78e-10\\
   $\|\tilde A_r-A_k\|_2$&1.00e-5 &  1.00e-5 &  6.55e-15 &  2.67e-14 &  6.57e-15  & 1.31e-14 &  9.79e-14\\\hline\hline
$q=3$&$s_0=5$&$s_0=10$&&$b=5$&$b=10$&$b=5$&$b=10$\\\hline
 time(seconds)&0.062 &  0.047&  0.079 & 0.015 &  0.015 &  0.016 &  0.015\\
crank($r$)&10&10&10&10&10&10&10\\
 $\|I-Q^TQ\|_2$& 2.80e-13 &  4.75e-12  & 1.34e-15 &  1.34e-15 &  2.01e-15 &  1.33e-15 &  1.23e-15\\
  range error&  3.89e-13 &  3.74e-13 &  3.67e-13 &  3.70e-13 &  3.74e-13 &  3.77e-13 &  4.01e-13\\
 $\|\tilde A_r-A_k\|_2$&1.00e-05 &  1.00e-05 &  3.80e-15 &  3.99e-15 &  3.89e-15 &  3.90e-15 &  3.91e-15\\
			\midrule
		\end{tabular}
\end{table}

Let $r$ denote the computed rank (crank),  ${\cal R}(Q)$ with $Q\in {\mathbb R}^{m\times r}$ and  $\tilde A_r=QQ^TA$ denote  the computed numerical range and low-rank approximation, respectively.
With  the generated factorization $A=U\Sigma V^T$ and the partition $U=[U_{(k)}\quad Y]$,    the deviation degree
\[
{\bf d}({\cal R}(Q), {\cal R}(U_{(k)}))=\|Y^TQ\|_2
\]
measures the deviated range error of the computed range space ${\cal R}(Q)$ with respect to ${\cal R}(U_{(k)})$.

For Type I matrices, we see from Table \ref{table2} that all methods can determine the right numerical rank,  while in view of $\|\tilde A_r-A_k\|_2$, the rank-$k$  matrix approximation via {\sf lansvd}
can not be as accurate as the ones from   {\sf larank}\footnote{http://www.netlib.org/numeralgo/na46.zip}, {\sf blarank} and {\sf sblarank} methods.  That is because {\sf lansvd} can not compute the singular values very accurately, say, e.g., for $q=2$  the computed singular value $\tilde \sigma_{k_1}=9e-5$ is of low accuracy.
The orthogonality of $Q$ via  {\sf larank} and {\sf sblarank}  is  more satisfactory than that from {\sf lansvd}, while they provide range error of lower accuracy, especially when $q=1$ and $q=2$. The accuracy of the range error is enhanced when $q$ is increased to 3. This means the increase of $q$ avoids $Q$ capturing the excessive information associated with tail singular values of $A$.
Among all methods, the efficiency of {\sf blarank} and {\sf sblarank} is at least ten times higher than that of {\sf larank}.

Similar conclusions apply to Type II matrices.  The exception in Table \ref{table3}  is that in some situations,  {\sf larank}, {\sf blarank} fail to reveal the right numerical rank,
which results in bad estimates of orthogonality or range error. Empirical results also show that the power scheme parameter $q=2$ is sufficient to give a high-quality of  $\|\tilde A_r-A_k\|_2$.

\renewcommand\tabcolsep{3.5pt}
\begin{table}[!htbp]\label{table3}
	\centering
	\caption{The numerical results for Type II matrices}
	\begin{tabular}{c|cc|c|cc|cc}
			\midrule
		{Type II}&\multicolumn{2}{c|}{\sf lansvd}&{\sf larank }& \multicolumn{2}{|c|} {\sf blarank}&\multicolumn{2}{c} {\sf sblarank}  \\\cline{1-8}
$q=1$&$s_0=10$&$s_0=20$&&$b=10$&$b=20$&$b=10$&$b=20$\\\hline
time(seconds)&0.483&0.327 &29.67&0.047&0.047&0.063&0.031\\
crank($r$) &20  &  20   &   800  &    22   &   29   &   20  &    20\\
 $\|I-Q^TQ\|_2$& 1.45e-10   &  9.00e-11 &    7.80e+2  &   4.47e-15  &   4.45e-15  &   5.89e-15   &  5.96e-15\\
 range error& 3.46e-9  &   2.70e-9  &   2.12e-6   &  1.00  &   1.00  &   9.63e-7   &  8.46e-6\\
$\|\tilde A_r-A_k\|_2$&1.00e-9  &   1.00e-9  &  7.54e-5  &   8.59e-11  &   8.88e-11 &   1.74e-14  &   2.47e-13\\\hline\hline
$q=2$&$s_0=10$&$s_0=20$&&$b=10$&$b=20$&$b=10$&$b=20$\\\hline
time(seconds)&0.499  &  0.344  &   27.65 &    0.082  &   0.066 &    0.069  &  0.056\\
crank($r$) & 20   &   20  &   800  &    22  &    24  &    20  &    20\\
 $\|I-Q^TQ\|_2$&5.96e-11  &   7.48e-11  &   78.00   &  3.02e-15  &   9.88e-15  &   4.55e-15 &    2.10e-15\\
 range error& 2.82e-9  &   2.76e-9  &   4.14e-6  &   1.00  &   1.00  &   7.97e-8   &  2.75e-9\\
$\|\tilde A_r-A_k\|_2$&1.00e-9  &   1.00e-9  &   1.39e-4  &   8.37e-11  &   8.63e-11  &   2.58e-15  &   2.41e-15\\\hline\hline
$q=3$ &$s_0=10$&$s_0=20$&&$b=10$&$b=20$&$b=10$&$b=20$\\\hline
 time(seconds)&0.484  &0.359&  28.85 &0.109&0.078&0.112&0.094\\
crank($r$) &20  &  20 &  800 &   24   & 25  &  20  &  20\\
 $\|I-Q^TQ\|_2$&3.29e-11 &  8.23e-11 &  7.81e+2 & 7.91e-15 &  7.62e-15 &  5.57e-15 &  1.84e-15\\
  range error& 1.72e-9&   1.92e-9 &  1.87e-6 &  1.00&  1.00&  9.91e-7 &  1.72e-9\\
 $\|\tilde A_r-A_k\|_2$&1.00e-9  &  1.00e-9 &   8.97e-11  &  8.30e-11 &   1.80e-14 &   4.35e-15 &   1.24e-15\\
			\midrule
		\end{tabular}
\end{table}

We continue considering Type II matrix $A$. Let $r=b(\ell-1)+s$ be the computed rank via  {\sf blarank} or {\sf sblarank} algorithm, and  $Q=[Q_1~\ldots~ Q_{\ell-1}~ Q_\ell]$ be the corresponding basis matrix.
We now  display the quality of $Q$ by estimating the singular values $\hat\sigma_j$ of $A$  via those of $Q_1^TA, Q_{2}^TA, \ldots,  Q_{\ell}^TA$, respectively. Here the singular value $\sigma_i(Q_j^TA)$ is expected to be an   approximation of  $\sigma_{b(j-1)+i}$, according to the estimate in \eqref{eq4}.
In Figure \ref{Fig4.0}, we depict the relative errors $|\hat\sigma_i-\sigma_i|/\sigma_i$ for $i=1,2,\ldots, r$, with respect to different values of $b$ and $q$.
We note that the neighbouring singular values   ratio $\sigma_{i+1}/\sigma_i$ of $A$  is 0.1 for $i=1,\ldots,4$, this ratio becomes 0.01 for $i=5$, and it increases to
$0.7499$ for $i=6,\ldots, 19$ and then decreases to $0.01$  for $i=20$. When $i>20$, the ratio remains 0.985.

The depicted results show that for each algorithm, the first 5 computed singular values have accuracy close to machine precision, and the accuracy  decreases in the sequent 10 singular values whenever $b=5, 10$ or $15$.  This is because of the big singular values ratio $\sigma_{bs+1}/\sigma_{bs}$ for $bs\le 15$. For the approximate singular values, a bigger $b$ or $q$ might generate  approximations with higher-accuracy, especially for $b=15$ and the  singular values $ \sigma_i(Q_1^TA)$ far from $\sigma_b(Q_1^TA)$, even the neighbouring singular values ratio $\sigma_{16}/\sigma_{15}$ is close to 0.75.  These observations coincide with the proposed theory in  \eqref{eq4}.

We also note that for some specific $b$, say $b=5, 10$ the accuracy of computed singular values $\hat\sigma_{16}\sim\hat\sigma_{20}$ is getting enhanced for both algorithms, and the computed singular value $\hat\sigma_{16}$ has bad accuracy than $\hat\sigma_{20}$ even $\sigma_{16}$ is far from $\sigma_{20}$ among  $\sigma_{16}\sim\sigma_{20}$.
This might because the singular value ratio $\sigma_{21}/\sigma_{20}$ has been decreased to 0.01.
In {\sf blarank}, the computed singular values $\hat \sigma_{i} (i\ge 20)$  have much worse accuracy, and the computed rank is some times greater than the true numerical rank, especially for   $\hat\sigma_{20}$ with $b=15,q=2$, even the singular values ratio $\sigma_{21}/\sigma_{20}$ is as small as $0.01$.
The accuracy and computed rank are greatly enhanced
in {\sf sblarank}. This happens partly because  we use twice-orthogonalization operator {\sf reorth2} in {\sf sblarank}, while  one-orthogonalization of $Q_j$ against $Q_{[j-1]}$ in {\sf blarank} is not enough
to ensure $Q_j$ is highly orthogonal to $Q_{[j-1]}$ within the roundoff error, and in $Q_j$ there might exist  some information  associated with ${\cal R}(Q_{[j-1]})$ that makes $Q_j^TA$
generate  singular values greater than the threshold $\theta$.

\begin{figure}
\centering
\includegraphics[width=0.9\textwidth,height=6cm]{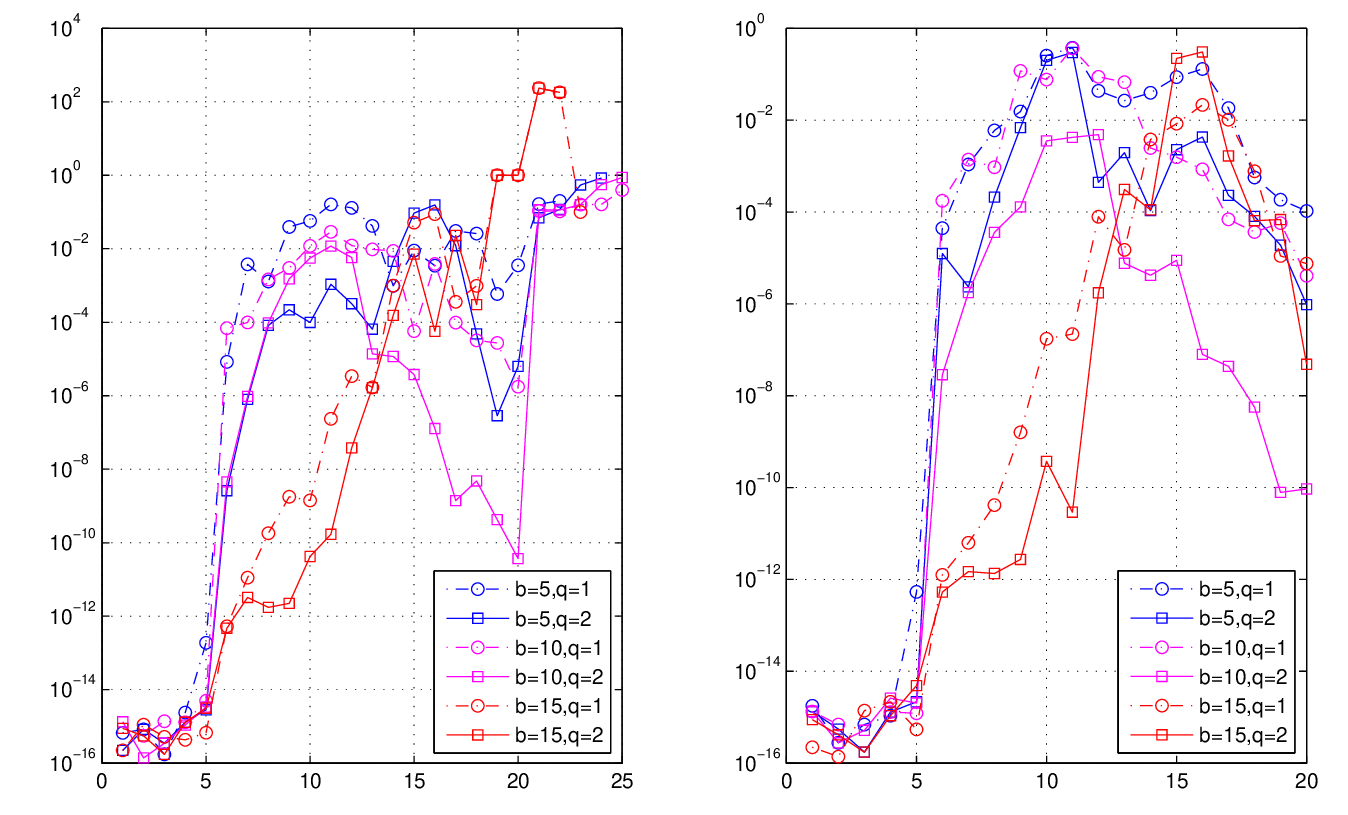}
\caption{\it The relative errors of  computed singular values $\{\hat \sigma_j\}_{j=1}^k$ with respect to $j$ for different $b$ and $q$. The left figure corresponds to the {\sf blarank} algorithm and the right is for {\sf sblarank}.} \label{Fig4.0}
 \end{figure}

\subsection{Applications in Information retrieval}
We test our algorithm with two real data. One is  from a standard document collection Cranfield \cite{co,llz2}. The collection provides about 30000 terms selected from 1400 documents, which can be described by a term-by-document matrix $A=[a_{ij}]\in {\mathbb R}^{m\times n}$, where
\[
a_{ij}=\mbox{the number of times terms $T_i$ occurs in document $D_j$}.
\]
In practical applications, the matrix  $A$ is usually contaminated with noise caused by the presentation style, ambiguity in the use of vocabulary, etc.
A method called {\it latent semantic indexing} \cite{bf,llz,zs} method uses key words to find the most relevant documents from the library database, which relies on the computation of low rank approximation for large matrices. In practical, when a set of key words is submitted, a query vector ${\bf q}$ is formed. With the rank-$k$ approximation  $QQ^TA$ of $A$, we can characterize the library database with a low-dimension subspace ${\cal R}(Q)$. In the new coordinate system,   the query vector ${\bf q}$ and each document $D_j$ can be represented by
$\hat {\bf q}=Q^T{\bf q}$ and $w_j=Q^Ta_j$, respectively.
 Moreover,
the cosine of the angle between the  query vector ${\bf q}$  and the document $D_j$ can be approximated as
\[
\cos\alpha_j={{\bf q}^Ta_j\over \|{\bf q}\|_2\|a_j\|_2}\approx {{\hat {\bf q}}^Tw_j\over \|q\|_2\|w_j\|_2}.
\]
  The larger magnitude of $\cos\alpha_j$ is, the more relevant the document $D_j$  relates to the query ${\bf q}$. Note that $\hat {\bf q}, w_j$ are vectors of length $k$, the computation is more economical in terms of $k$-dimensional vector inner product. In our test, we choose
the first 3000 terms from Cranfield  to form a term-by-document matrix $A\in {\mathbb R}^{3000\times 1400}$, whose largest 80 singular {values range}  from 68.61 to 6.12.

The other real data is the standard  color image ``Mandril'' with $512\times 512$ pixels. We construct a $1536\times 512$ real matrix $A$ by using the pixel values in red, green and blue channels of the image. Instead of storing the whole $m\times n$ matrix,  using a low-rank representation of the matrix can reduce the storage substantially with acceptable quality of the image. For example, in SVD, we use
 the dominant part $\sigma_1 u_1v_1^T+\cdots+\sigma_k u_kv_k^T$  of the matrix $A$ to compress the storage from $mn$ to
$(m+n)k$, while  for {\sf larank} and {\sf sblarank} algorithms, we only need to compute $Q$ and $Q^TA$, and then store them with the compression ratio (cratio) as
${mn\over (m+n)r}$, where $r$ denotes the computed rank.

\begin{figure}
\centering
\includegraphics[width=0.8\textwidth,height=5cm]{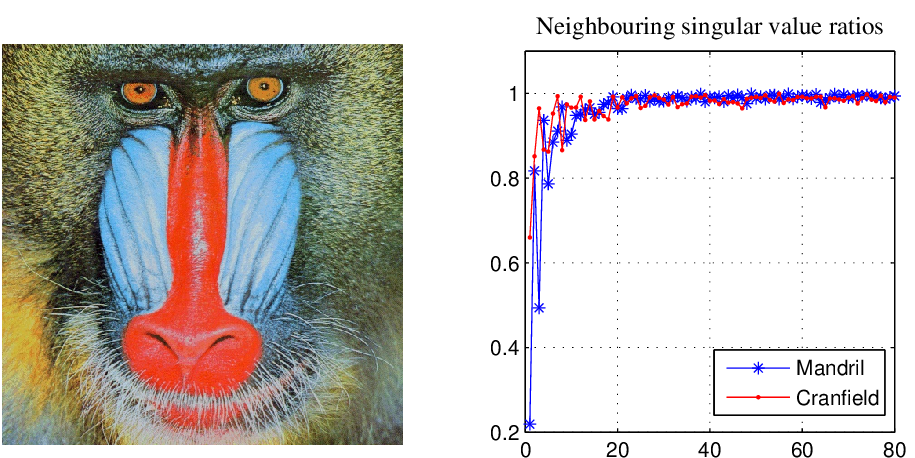}
\caption{\it The original color image of Mandril, neighbouring singular value ratio of tested matrices    }\label{Fig4.1}
 \end{figure}

In Figure \ref{Fig4.1}, for  ``Mandril'' and ``Cranfield'',  we show the original image of ``Mandril'',  and display the first 80 neighbouring singular value ratios $\sigma_{i+1}/\sigma_i$ of each matrix.
 It is seen that the ratios are generally greater than 0.9.
In Table \ref{table2}-\ref{table3}  we try to retrieve the information with  respect to the given threshold $\theta$, where ${\theta\over \|A\|_2}$ controls the percentage of the unretrieved information.

Let  ${\rm relerror}={\|\tilde A_r-A\|_2/\|A\|_2}$
 be the relative   error of the matrix approximation.
 In Table \ref{table4}, the tabulated results  show that {\sf larank} gives good estimates of the computed rank and relative error, but costs more time than {\sf sblarank}.
Perhaps it is because of the large singular value ratios that affect the accuracy of ${Q}$,
the {\sf sblarank} methods sometimes can not give   good estimates of the computed rank and relative error, especially for $\theta=20\%\|A\|_2$.   This phenomenon is also reflected in
Table \ref{table5}, where for the matrix ``Mandril'' and given threshold, different methods give different estimates of the computed rank and compression ratio for the image.
In spite of this, the compressed images via {\sf blarank} and {\sf sblarank} with $b=10, q=2$ in Figure \ref{Fig4.2} show that they have similar visual quality   for the same threshold.

\renewcommand\tabcolsep{2.0pt}
\begin{table}[!htbp]\label{table4}
	\centering
	\caption{The numerical results for Cranfield}
	\begin{tabular}{c|c|ccc|ccc|cc}
			\midrule
		Threshold &parameters&\multicolumn{3}{c|}{\sf sblarank}&\multicolumn{3}{c}{\sf larank }&\multicolumn{2}{c}{\sf svd} \\\cline{3-10}
 $\theta$&&crank&time&relerror&{crank}&time&relerror&crank&relerror\\\hline
\multirow{4}{*}{$10\%\|A\|_2$}&
$b=10, q=2$& 68&0.685&0.121&\multirow{2}{*}{71}&\multirow{2}{*}{9.346}&\multirow{2}{*}{0.102}&\multirow{4}{*}{70}&\multirow{4}{*}{0.0995}\\
&$b=10,q=3$&68&0.734&0.115&&&\\\cline{2-8}
&$b=20,q=2$&70&0.406&0.107&\multirow{2}{*}{68}&\multirow{2}{*}{9.497}&\multirow{2}{*}{0.105}\\
&$b=20,q=3$&70&0.640&0.103&&&\\\cline{1-10}
\multirow{4}{*}{$20\%\|A\|_2$}&
$b=10, q=2$& 19&0.172&0.244&\multirow{2}{*}{24}&\multirow{2}{*}{2.892}&\multirow{2}{*}{0.210}&\multirow{4}{*}{25}&\multirow{4}{*}{0.1949}\\
&$b=10,q=3$&24&0.327&0.204&&&\\\cline{2-8}
&$b=20,q=2$&26&0.202&0.197&\multirow{2}{*}{24}&\multirow{2}{*}{2.919}&\multirow{2}{*}{0.208}\\
&$b=20,q=3$&25&0.265&0.196&&&\\\cline{1-10}
			
		\end{tabular}
\end{table}

\renewcommand\tabcolsep{4.0pt}
\begin{table}[!htbp]\label{table5}
	\centering
	\caption{The numerical results for Mandril ($q=2,b=10$)}
	\begin{tabular}{c|ccc|crc|crc}
			\midrule
	\multirow{2}{*}{Threshold $\theta$} &\multicolumn{3}{|c|}{\sf svd}&\multicolumn{3}{|c|}{\sf larank}&\multicolumn{3}{|c}{\sf sblarank }\\\cline{2-10}
 &time&crank&relerror&{time}&cratio&relerror&time &cratio&relerror\\\hline
{$5\%\|A\|_2$}&0.4&9&4.52e-2&0.19&54.9:1&5.80e-2&0.05&38.4:1&4.65e-2\\
{$2\%\|A\|_2$}&0.4&50&1.98e-2&1.51&7.7:1&2.03e-2&0.09&8.2:1&2.27e-2\\\hline			
		\end{tabular}
\end{table}

\begin{figure}\label{Fig4.2}
\centering
~~~~~~\includegraphics[width=0.9\textwidth,height=7cm]{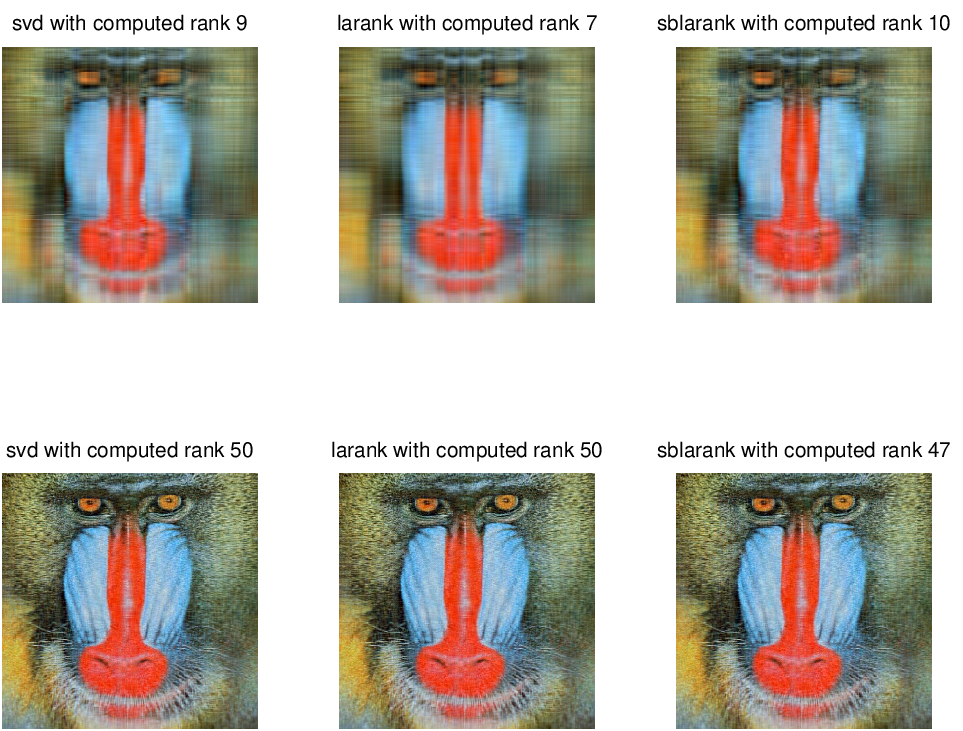}
\caption{\it Color image compression with given threshold. The first row corresponds to the threshold $5\%\|A\|_2$, and the second row corresponds to the threshold $2\%\|A\|_2$. }
 \end{figure}

 \subsection{Application in background modelling model}

 Background estimation from video sequences is a classic computer vision problem.
 It can be regarded as a  low rank approximation problem, especially when the camera motion is
presumably static, the background image sequence can be modeled as a low-dimensional
linear subspace, while  the foreground layer which is relatively sparse comparing to
the slowly changing background layer can be modeled as a sparse outlier component of
the video sequence.

Classical principal component analysis (PCA) is not robust
to the presence of sparse outliers in the data. The use of $\ell_1$ norm in the robust principal component analysis (RPCA) \cite{clmw}:
\begin{equation}
\min_{L, S}\|L\|_*+\lambda \|S\|_1, \quad s.t.\quad A=L+S,\label{rpca}
\end{equation}
can eliminate the weakness of PCA in separating the sparse outliers, where $\lambda$ is a weighting parameter, $A$ is the matrix to store each frame of the video sequence
in the vectorized form, the low-rank matrix $L$ is assumed to capture the background
information, and $S$ is  used to capture the sparse moving objects in the video sequence. The augmented Lagrange multipliers (ALM) method \cite{be}  solving  \eqref{rpca} involves three steps in the $j$th loop:

1: $L_{j+1}={\sf shrink}'(A-S_j+\mu_j^{-1}Y_j, \mu_j^{-1})$;

2: $S_{j+1}={\mathscr S}_{\lambda \mu_j^{-1}}(A-L_{j+1}-S_{j})$;

3: $Y_{j+1}=Y_j+\mu_j (A-L_{j+1}-S_{j+1})$, $\mu_{j+1}=\rho\mu_j$, \\
where  $S_0$ is the initial guess of the sparse outlier in the video,   $Y_j$ and $\mu_j$ are multipliers of Lagrange in the Lagrange function
\[
{\cal L}(L, S, Y, \mu)=\|L\|_*+\lambda \|S\|_1+{\rm trace}(Y^T(A-L-S))+{\mu\over 2}\|A-L-S\|_F^2.
\]
 Note that the shrinking operator {\sf shrink}($W, \mu_j^{-1}$) in the iteration  needs to compute the SVD of a matrix $W$ and then shrink the singular values by $\mu_j^{-1}$.
Instead of SVD, we use the {\sf sblarank} algorithm to compute an approximation $W\approx Q(Q^TW)$ and then obtain a thin  two-sided orthogonal factorization
\[
W\approx Q(Q^TW)=QLP^T
\]
 by the thin  QR
factorization of $W^TQ=PL^T$.   We shrink the matrix $W$ by using  the entrywise soft-thresholding operator ${\mathscr S}_{\tau}(\cdot)$ (see \eqref{0002}) to $L$:
\[
{\sf shrink}'(W, \mu_j^{-1})=Q{\mathscr S}_{\mu_j^{-1}}(L)P^T.
\]

\begin{center}
\begin{figure}\label{Fig4.3}
~~~~~~~~\includegraphics[width=0.9\textwidth,height=7cm]{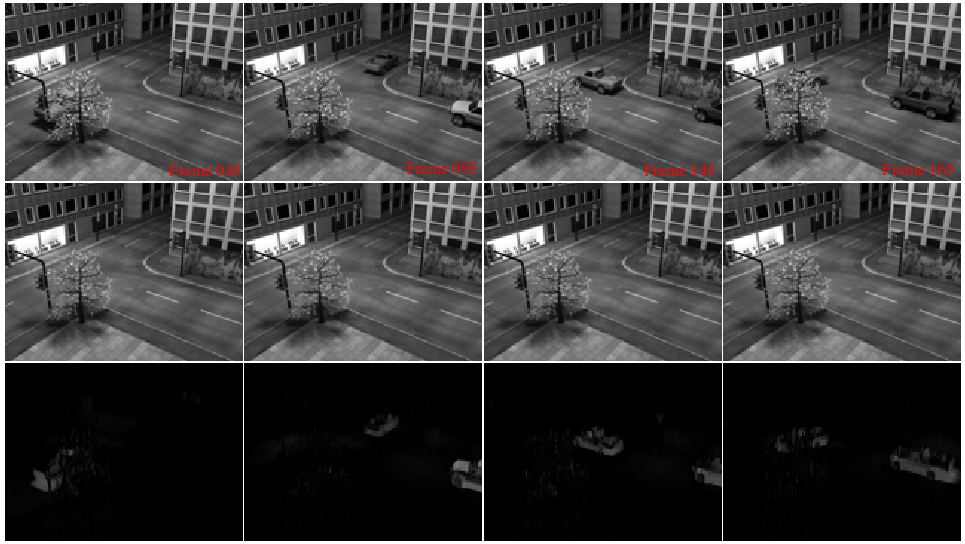}
\caption{\it Background and the foreground moving cars estimated by svd-based RPCA. The first row corresponds to the original frame, the second row is the background estimation, and the third row corresponds to the moving cars in the video.}
 \end{figure}
 \end{center}

\begin{center}
 \begin{figure}\label{Fig4.4}
 \centering
~~~~~~~\includegraphics[width=0.9\textwidth,height=11cm]{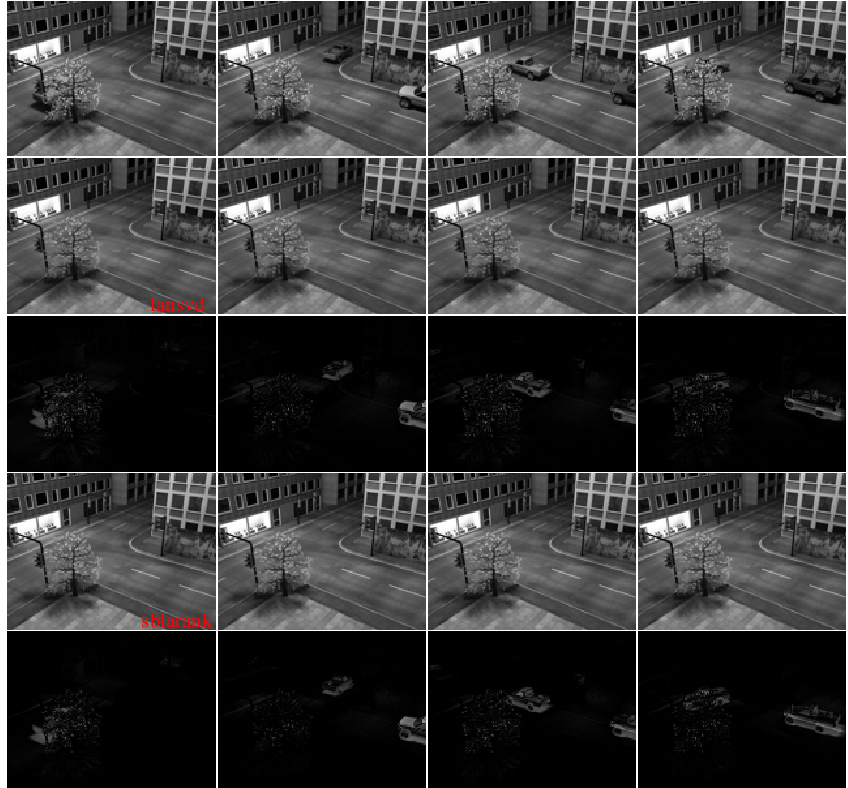}
\caption{\it  Background and the foreground moving cars estimated by  {\sf lansvd}-based RPCA (in rows 2 and 3) and {\sf sblarank}-based RPCA (in rows 4 and 5). }
 \end{figure}
 \end{center}

We use the basic scenario of Stuttgart synthetic video data set  \cite{bh} to estimate the background and the moving cars in the video.  The first 300 frames in the Basic scenario are chosen, and each frame of size
$600\times 800$ is resized to $m\times n$ with $(m,n)=(120,160)$, and then is stored as a column vector  in $A$, where the entries of $A$ range from 0 to 255.

We compare the RPCA algorithm via ${\sf svd}$, {\sf lansvd} and ${\sf sblarank}$, respectively. The maximum iteration number is 100, and other input parameters $\lambda=(mn)^{-1/2}$, $\mu_0=1e-3$ and $\rho=1.1$. For ${\sf sblarank}$, we choose $b=10$, $q=0$, since the matrix is expected to exhibit
low rank due to the spatial correlation within the background of the video.  The iteration stops if
\[
{\rm relerror}={\|A-L_j-S_j\|_F\over \|A\|_F}< 9e-5.
\]

We pick some special frames in the video, for example, in Frame 045/165, there is a car behind a tree, and in Frame 095/145,
  a car is driving into the monitoring area, while a car is leaving this area at the far end. We test whether
 the candidate methods can provide  a  high-quality background estimation and  capture the moving objects at the same time.
The modelling results are shown in Figures \ref{Fig4.3}-\ref{Fig4.4}. The {\sf sblarank}-shrinking method can provide similar visual   quality of the background to the one from
{\sf svd}-shrinking  technique, and the results in  Table \ref{table6} show that the
{\sf sblarank}-based method needs more iterations but  is still the most efficient, due to the
proposed rank prediction scheme, while {\sf lansvd} is not flexible in predicting the dimension of the leading singular space, and it costs more time in the process of dimension prediction.

\renewcommand\tabcolsep{25.0pt}
\begin{table}[!htbp]\label{table6}
	\centering
	\caption{The comparison of the methods  for background estimation of Stuttgart  videos}
	\begin{tabular}{c|ccccrccrc}
			\midrule
	 & {\sf svd}& {\sf lansvd}& {\sf sblarank }\\\hline
 time & 57.76 &120.51& 21.65\\
 iterations & 28 &30& 58\\
 relerror & 8.21e-5& 8.14e-5& 8.45e-5 \\\hline			
		\end{tabular}
\end{table}

\section{Conclusion}
This paper proposed  efficient adaptive randomized rank-revealing algorithms for fixed-threshold low-rank approximation problems.
Based on a deflation process, the algorithm generates basis matrix  $Q:=[Q~ Q_\ell]$ of $A$ block by block,
and each block can align well  with those subblocks in the left singular matrix of $A$, provided the singular value gap is big and the twice-orthogonalization
scheme is used in the algorithm. We present provable results for the rank deduce in the deflation process, as well as the estimates for
the singular values of $Q_\ell^TA$   and the  matrix
approximation error via spectral and Frobenius norms.
Although the proposed algorithm is not a kind of singular value thresholding (SVT) algorithm, it can be viewed as an approximation of SVT  and be used to solve SVT-based optimization problem.
Applied to synthetic data and some applications in information retrieval, background estimation of videos, the proposed algorithms are efficient and effective than Lanczos-based algorithms and a rank-revealing algorithm proposed by Lee, Li and Zeng.


\begin{thebibliography}{zz}



\bibitem{ber} \textsc{D. S. Bernstein}, {\it Matrix Mathematics: Theory, Facts, and Formulas}, 2nd edn. Princeton University Press,
Princeton and Oxford (2009)

\bibitem{bf} \textsc{M. W. Berry and  R. D. Fierro}, {\it Low-rank orthogonal decompositions for information retrieval
applications}, Numer. Linear Algebra Appl., 3 (1996), pp. 301-328.


\bibitem{be} \textsc{D. Bertsekas}, {\it Constrained Optimization and Lagrange Multiplier  Method}, Academic Press, (1982)



\bibitem{bh} \textsc{S. Brutzer, B. H$\ddot{\rm o}$ferlin,  and G.  Heidemann},  {\it Evaluation of background subtraction techniques for video surveillance}.
IEEE Compu. Vis. Patt. Recog.,     { 32}:14 (2011), pp. 1937-1944.


\bibitem{ccs} \textsc{J. F. Cai, E. J. Cand$\grave{\rm e}$s, and Z.W. Shen},  {\it A singular value thresholding algorithm for matrix completion}. SIAM J. Optim.,  20:4 (2010),  pp. 1956-1982.

\bibitem{clmw} \textsc{E. J. Cand$\grave{\rm e}$s, X. Li, Y. Ma,  and J. Wright}, {\it Robust principal component analysis?}  J.
 Assoc. Comput. Mach., 58 (2011), pp. 11:1-11:37.



\bibitem{ch} \textsc{T. F. Chan}, {\it Rank revealing QR factorizations}, Linear Algebra Appl., 88/89 (1987), pp. 67-82.

\bibitem{ci} \textsc{S. Chandrasekaran and I. C. F. Ipsen}, {\it On rank-revealing QR factorisations}, SIAM J. Matrix
Anal. Appl., 15 (1994), pp. 592-622.


\bibitem{co} Cornell SMART System, ftp://ftp.cs.cornell.edu/pub/smart.

\bibitem{ge} \textsc{M. Gu and S.C. Eisenstat}, {\it Efficient algorithms for computing a strong rank-revealing QR
factorization}, SIAM J. Sci. Comput., 17 (1996), pp. 848-869.














\bibitem{fh} \textsc{R. D. Fierro and P. C. Hansen}, {\it Low-rank revealing UTV decompositions}, Numer. Algorithms,
15 (1997), pp. 37-55.


\bibitem{hh} \textsc{R. D. Fierro, P. C.  Hansen, and P. S. K. Hansen}, {\it UTV tools: MATLAB templates for rank-revealing UTV
decompositions}, Numer. Algorithms,  20 (1999), pp. 165-194.











\bibitem{gv2}  \textsc{G. H. Golub and C. F. Van Loan}, {\it Matrix Computations(4ed.)},  Johns Hopkins University Press, Baltimore (2013)

\bibitem{gu} \textsc{M. Gu}, {\it Subspace iteration randomization and singular value problems}, SIAM J. Sci. Comput. 37 (2015),  pp. A1139-A1173.





\bibitem{hmt} \textsc{N. Halko, P. G.  Martinsson, and J. A. Tropp}, {\it Finding structure with randomness: probabilistic algorithms for constructing approximate matrix
decompositions},  SIAM Rev. 53 (2011), pp. 217-288.

\bibitem{hp} \textsc{Y. P. Hong and C. T. Pan}, {\it Rank-revealing QR factorizations and the singular value decomposition},
Math.  Comput., 58 (1992), pp. 213-232.



\bibitem{la} \textsc{R. M. Larsen}, {\it PROPACK-Software for large and sparse SVD calculations}, http://sun.
stanford.edu/\textasciitilde rmunk/PROPACK/


\bibitem{llz}  \textsc{T. L. Lee,  T. Y. Li, and   Z. G.  Zeng}, {\it A rank-revealing method with updating, downdating, and applications.
Part II},  SIAM J. Matrix Anal. Appl. 31 (2009), pp. 503-525.


\bibitem{llz2} \textsc{T. L. Lee, T. Y. Li, and Z. G.  Zeng}, {\it RankRev: a Matlab package for computing
the numerical rank and updating/downdating}, Numer. Algorithms,   77 (2018), pp. 559-576.


\bibitem{lwm} \textsc{E. Liberty, F. Woolfe, P. Martinsson, V. Rokhlin, and M. Tygert}, {\it Randomized
algorithms for the low-rank approximation of matrices},  Proceedings of the
National Academy of Sciences, 104(2007), pp. 20167-20172.


\bibitem{llj} \textsc{Q. H. Liu, S. T. Ling, and Z. G. Jia}, {\it Randomized quaternion singular value decomposition for low-rank matrix
approximation}, SIAM J. Sci. Comput., 44:2 (2022), pp. A870-A900.


\bibitem{ma} \textsc{M. W. Mahoney}, {\it Randomized algorithms for matrices and data}, Found. Trends Mach. Learn., 3(2011), pp. 123-224.





\bibitem{mart} \textsc{P. G. Martinsson, V. Rokhlin, and M. Tygert}, {\it A randomized algorithm for the decomposition
of matrices}, Appl. Comput. Harmon. Anal., 30 (2011),  pp. 47-68.




\bibitem{mv} \textsc{P. G. Martinsson and S. Voronin}, {\it A randomized blocked algorithm for efficiently computing rank-revealing factorization of matrices},
 SIAM J. Sci. Comput., 38(2016), pp. S485-S507.


\bibitem{pan} \textsc{C. T. Pan},  {\it On the existence and computation of rank-revealing LU factorizations},  Linear
Algebra   Appl., 316 (2000), pp. 199-222.

\bibitem{pt} \textsc{C. T. Pan and P. T. P. Tang}, {\it Bounds on singular values revealed by QR factorizations}, BIT, 39 (1999), pp. 740-756.


\bibitem{rlb} \textsc{H. Ren, R-R Ma, Q. H. Liu, and Z.-J. Bai}, {\it Randomized quaternion QLP decomposition for low-rank approximation}, J. Sci. Comput., 92 (2022), 80.


\bibitem{rxb} \textsc{H. Ren, G. Y. Xiao, and Z.-J. Bai}, {\it  Single-pass randomized QLP decomposition for low-rank approximation},
Calcolo,  59 (2022), 49.



\bibitem{sa} \textsc{A. K. Saibaba}, {\it Randomized subspace iteration: Analysis of canonical angles and unitarily invariant norms}, SIAM J. Matrix Anal. Appl.,
40:1 (2019), pp. 23-48.




\bibitem{wlr} \textsc{F. Woolfe, E. Liberty, V. Rokhlin, and M. Tygert}, {\it A fast randomized algorithm
for the approximation of matrices},  Appl.  Comput. Harm.
Anal., 25:3(2008), pp. 335-366.

\bibitem{ws} \textsc{K. Wu and H. Simon}, {\it Thick-restart Lanczos method for large symmetric eigenvalue problems},
SIAM J. Matrix Anal. Appl., 22 (2000), pp. 602-616.

\bibitem{ygl} \textsc{W. J. Yu, Y. Gu, and Y. H. Li}, {\it Efficient randomized algorithms for the fixed-precision low-rank matrix approximation}, SIAM J. Matrix Anal. Appl.,
 39 (2018), pp. 1339-1359.

\bibitem{zs} \textsc{H. Zha and H. D. Simon}, {\it On updating problems in latent semantic indexing}, SIAM J. Sci.
Comput., 21 (1999), pp. 782-791.


\end{thebibliography}
\end{document}